\documentclass{article}
\usepackage{amsmath}
\usepackage{amsfonts}
\usepackage{amssymb}
\usepackage{amsthm}
\usepackage{mathtools}
\usepackage{graphicx}%
\providecommand{\U}[1]{\protect\rule{.1in}{.1in}}
\providecommand{\U}[1]{\protect\rule{.1in}{.1in}}

\newcommand{\s}{s}

\newtheorem{theorem}{Theorem}

\newtheorem{lemma}[theorem]{Lemma}

\everymath{\displaystyle}

\usepackage{graphics,xcolor}
\newcommand\grdel[1]{{\color{black}{#1}}}

\newcommand\gui[1]{{\color{black}{#1}}}

\newcommand\amf[1]{{\color{black}{#1}}}

\begin{document}
  
\title{Resonance in rarefaction and shock curves: \\local analysis and numerics of\\ the continuation method.}

\author{\ A. C. Alvarez
\thanks{Instituto Nacional de Matem\'{a}tica Pura e Aplicada, Estrada Dona
Castorina 110, 22460-320 Rio de Janeiro, RJ, Brazil. E-mail:
amaury@impa.br} , \ G.T. Goedert \thanks{Instituto Nacional de
Matem\'{a}tica Pura e Aplicada, Estrada Dona Castorina 110, 22460-320 Rio de
Janeiro, RJ, Brazil. E-mail: ggoedert@impa.br}, \ D. Marchesin\thanks{Instituto Nacional de
Matem\'{a}tica Pura e Aplicada, Estrada Dona Castorina 110, 22460-320 Rio de
Janeiro, RJ, Brazil. E-mail: marchesi@impa.br}
}

\maketitle

\begin{abstract}

In this paper we describe certain crucial steps in the development of an algorithm for finding the Riemann solution in systems of conservation laws. We relax the classical hypotheses of strict hyperbolicity and genuine nonlinearity of Lax. First, we present a procedure for continuing wave curves beyond points where characteristic speeds coincide, i.e., at wave curve points of maximal codimensionality.	
This procedure requires strict hyperbolicity on both sides of the coincidence locus. Loss of strict hyperbolicity is regularized by means of a Generalized Jordan Chain, which serves to construct a four-fold submanifold structure on which wave curves can be continued. Second, we analyze the case of loss of genuine nonlinearity. We prove a new result: the existence of composite wave curves when the composite wave traverses either the inflection locus or an anomalous part of the non-local composite wave curve. In this sense, we find conditions under which the composite field is well defined and its singularities can be removed, allowing use of our continuation method. Finally, we present numerical examples for a non-strictly hyperbolic system of conservation laws.	
	
\end{abstract}

\tableofcontents

\section{Introduction}

In this work, we develop algorithms for finding analytical and numerical the Riemann solution while relaxing the requirements of Lax \cite{lax1957hyperbolic}, namely strict hyperbolicity and genuine nonlinearity hypotheses. We study shock and rarefaction curves arising in systems of conservation laws in which there are states whose characteristic speeds coincide. Indeed, examples of such models are found in \cite{key95x,alvarez2017analytical,alvarez2017analytical1}. Furthermore, we take into account composite rarefaction-shock wave curves in situations in which analytical difficulties were not {fully considered previously. In this new retrospective, we have found a theoretical approach in \cite{muller2001existence} and a joint analytical-computational treatment of \cite{dahmen2005riemann} restricted to the Euler equation. 

Wave curves are fundamental tools in the construction of Riemann solutions,  see \cite{liu1974riemann,liu1975riemann,azevedo1995multiple} and references therein. In this case, according to a celebrated result of Lax (\cite{lax1957hyperbolic}), these curves in phase space correspond to sequences of shock, rarefaction and composite wave curves. The above mentioned work required strict hyperbolicity and genuinely nonlinear characteristic fields to prove the existence and uniqueness of wave curve equivalent to a Riemann solution. A general procedure for the construction of wave curves, even where Lax's hypothesis are violated, is still needed.
A more complete description of the wave curve method can be found in \cite{lax1957hyperbolic,azevedo1995multiple,lambert2010riemann,dahmen2005riemann}, with theoretical justification provided by  \cite{liu1974riemann,liu1975riemann,issacson1992global,wendroff1972riemann} and reference therein.

 We study rarefaction and shock wave curves \amf{in a neighborhood of the coincidence locus} for a general system of conservation laws 
 \begin{equation}
 \displaystyle{\frac{\partial G(U)}{\partial t}+\frac{\partial F(U)}{\partial x}=0,}\label{eqt}
 \end{equation}
 where $U=U(x,t):\mathbb{R}\times\mathbb{R}^+\longrightarrow \Omega\subset \mathbb{R}^n$, the accumulation functions $G(U)=(G_1(U),\cdots,G_{n}(U)):\Omega \longrightarrow \mathbb{R}^{n}$, and the flux functions $F(U)=(F_1(U),\cdots,\\F_{n}(U)):\Omega \longrightarrow \mathbb{R}^{n}$ are known.

A Riemann problem consists of a Cauchy (initial value) problem governed by equations of type \eqref{eqt} with initial data 
\begin{equation}
U(x,t=0) = \left \{ \begin{matrix} U_L & \mbox{if } x < 0,
\\ U_R & \mbox{if } x > 0.\end{matrix}\right.
\label{eqrt2}
\end{equation}

 Riemann solutions correspond to states L or R that give rise to wave curves in the phase space describing the transition between all intermediate states. Rarefactions are continuous self-similar solutions of \eqref{eqt}, therefore they are represented by
\begin{equation}
U=\widehat{U}(\xi),\quad \text{with} \quad \xi=x/t.
\label{eqt2}
\end{equation}
Substituting \eqref{eqt2} into system \eqref{eqt} we obtain the rarefaction curve by using the solution of the generalized eigenvalue problem

\begin{equation}
Ar=\lambda Br, \quad \text{where} \quad A={\partial F}/{\partial U}, \quad B={\partial G}/{\partial U}.
\label{ge1}
\end{equation}
The eigenvector $r$ is parallel to $d\widehat{U}/d\xi$, so the rarefaction curves are tangent to the characteristic field given by the normalized eigenvector $r$.  We consider such a rarefaction curve in detail when it appears close to a coincidence locus of codimension one. We assume that system \eqref{eqt} is strictly hyperbolic on both sides of the locus, i.e. that generalized eigenvalues in \eqref{ge1} are real and distinct, see \cite{lax1957hyperbolic}. The case where an elliptic region appears was studied in a previous work \cite{mailybaev2008hyperbolicity}.

A shock wave is a traveling discontinuity in a (weak) solution of system \eqref{eqt} given by
\begin{equation}
U(x,t) = \left \{ \begin{matrix} U^- & \mbox{if } x < st,
\\ U^+ & \mbox{if } x > st,\end{matrix}\right.
\label{eqrt2a}
\end{equation}
where $s$ is a real constant called the shock speed. Solution \eqref{eqrt2a} is a piecewise constant weak solution to the Riemann problem  defined by \eqref{eqt} and \eqref{eqrt2} if these states satisfy the Rankine-Hugoniot
condition (\cite{Shearer89}):
\begin{equation}
F (U^-) - F (U^+) = s(G(U^-) - G(U^+)).
\end{equation}

In the case that  strict hyperbolicity is lost, we study the continuation of rarefaction and shock wave curves assuming that matrix $B$ in the generalized eigenvalues problem \eqref{ge1} is singular; the case of $B=I$ was studied for a simpler setting in \cite{mailybaev2008hyperbolicity}. We study the situation in which generalized eigenvalues $\lambda_1,\cdots,\lambda_n$ of \eqref{ge1} are real but there is a point $U_o$ where two characteristic speeds coincide, i.e. $\lambda_i(U_o)=\lambda_{i+1}(U_o)$. At this point, the corresponding rarefaction curves associated to the eigenvector fields $r_i$ and $r_{i+1}$ intersect and the generalized eigenproblem  \eqref{ge1} has an eigenvalue $\lambda_o$ of multiplicity two with only one associated eigenvector $r_o$. Thus, while applying the wave curve method, the following question arises: how to continue the rarefaction curves $\mathcal{R}_i$ and $\mathcal{R}_{i+1}$ beyond the point $U_o$? To answer this question, we must take into account  the behavior of eigenvalues at both sides of the coincidence manifold $\mathcal{E}$.

The existence of a sole eigenvector at intersection points of two rarefaction curves implies that the tangent space on the coincidence locus has at most dimension $n-1$. Therefore, there are not enough directions to continue two intersecting curves in general. Instead, we complete the dimension of tangent space up to $n$ by means of an appropriate system of coordinates in the neighborhood of the coincidence locus, which guarantees the continuation of curves. This method has sound theoretical basis and therefore serves to construct a Riemann solution between two states situated at different sides of a coincidence locus.
 
We also study genuine nonlinearity loss, i.e. points where there is a characteristic speed with null directional derivative along the vector field $r=r_i(U)$, i.e.  $\lambda^{\prime}_i(r)= \nabla \lambda_i \cdot r$ = 0, and the rarefaction curve generically stops. To cross this inflection locus, the rarefaction curve needs to be coupled with a characteristic shock curve. This corresponds to the construction of a composite curve, which arises in state space during the construction of the solution of a Riemann problem in a non-strictly hyperbolic system of conservation laws. We also study the case where inflection locus intersects states of coinciding eigenvalues. We address the related difficulties and propose a procedure that resolves each situation. 
	
The main issue addressed here is the justification of the construction of composite curve by certain continuation methods. The rarefaction curve $\mathcal{R}_k$ for a field $k$ consists of the integral curve along the properly oriented right eigenvector $r_k$ associated to eigenvalue $\lambda_k$, i.e. the parametrized curve is found as solution of the initial value problem
\begin{equation}
\label{eq:rarefaction_curve}
\frac{d \mathcal{R}_k}{d \xi}=r_k(R_k(\xi)), \quad \mathcal{R}_k(0)=U^-.
\end{equation}
Generically, a rarefaction curve is required to have monotone characteristic speed. The work \cite{wendroff1972riemann} of Wendroff proved that shock velocities were bounded by the characteristic velocities of rarefaction waves, leading to a transition from a shock to a rarefaction wave. In order to continue a wave curve past an inflection point (i.e. where $d \lambda_k/ d \xi =0$), Liu introduced in \cite{liu1975riemann} the concept of composite wave curve to couple the curves in state space; which was essential in for solving more general hyperbolic conservation laws. 

The work \cite{muller2001existence} studied a situation that was not considered in \cite{liu1975riemann}, proving the existence and uniqueness of the composite curve near of a special point of the corresponding characteristic field. Building upon these results, we show the existence of local and non-local composite wave curves near certain singularities. We also propose an algorithm for the construction of these curves. More details on composite wave curves can be found in \cite{furtado1991structural}. We argue that the procedure used in \cite{muller2001existence} is also applicable to our case in order to prove the existence of a unique non-local composite curve when the Implicit Function Theorem cannot be used.

The proposed algorithms were implemented and tested through a Riemann solver developed in Matlab together with an improved version of the RPN C++ library. From the numerical point of view, the construction of correct wave curves requires: an ODE solver with appropriate stopping criteria; contour plot subroutines to find Hugoniot loci; a continuation method and appropriate data structures to represent and manipulate   curves. 

The paper is organized as follows. Section \ref{chainA} presents the generalized Jordan chain, which we use to construct a new system of coordinates related to the Jacobian matrices of flux and accumulation functions. Then, this is used to describe a procedure to regularize the  singularity appearing when two characteristic speeds coincide at a submanifold of codimension one. Section \ref{proceM} presents an algorithm to compute the continuation of a rarefaction curve beyond the coincidence of characteristic speeds; we apply this method to a model presenting such a phenomenon as an example. In the Section \ref{comp} we present a construction of a composite curve in the classical sense, as explained in \cite{liu1975riemann}, but consider anomalous cases allowed by the non-strictly hyperbolic setting. We provide analysis justifying the construction of composite curves and an algorithm to be integrated to a Riemann solver. 
Our conclusions can be found in Section \ref{conclu}. Appendix \ref{apend1} presents formulas with the relationship between the generalized Jordan chain and the Jacobian of flux and accumulation functions. Moreover describes the formulas from versal deformation theory useful in numerical implementations.

\section{Singularity at the coincidence locus}
\label{chainA}
 In this paper, we assume that system \eqref{eqt} is strictly hyperbolic on both sides of the coincidence manifold and that the matrix $B$ in the generalized eigenvalue problem \eqref{ge1} is singular with linearly independent rows. We utilize a generalized Jordan chain to remove the singularities of the submanifold that appear when two rarefaction curves intersect. To do so, we require completing the dimension of the fundamental manifold up to $n$ at the coincidence locus (see \cite{issacson1992global}). 

This is achieved by lifting the solution curves from the state space $U$ to the higher dimensional extended space $(U,\lambda)$ (see Figure \ref{figFa}), where each eigenvalue defines a hypersurface. In the neighborhood of the  coincidence locus in the extended space, we construct a new coordinate system consisting of four charts. In these new coordinates, we apply versal deformation results to the generalized Jordan chain in order to recover a full set of directions necessary to continue the wave curves beyond the coincidence singularity. Our construction generalizes previous results in \eqref{ge1} for the case $B$ proportional to the identity matrix, the details of which can be found in \cite{mailybaev2008hyperbolicity,mailybaev2000transformation,seyranian2003multiparameter}.

\begin{figure}
	\centering
	\includegraphics[width=0.69\textwidth]{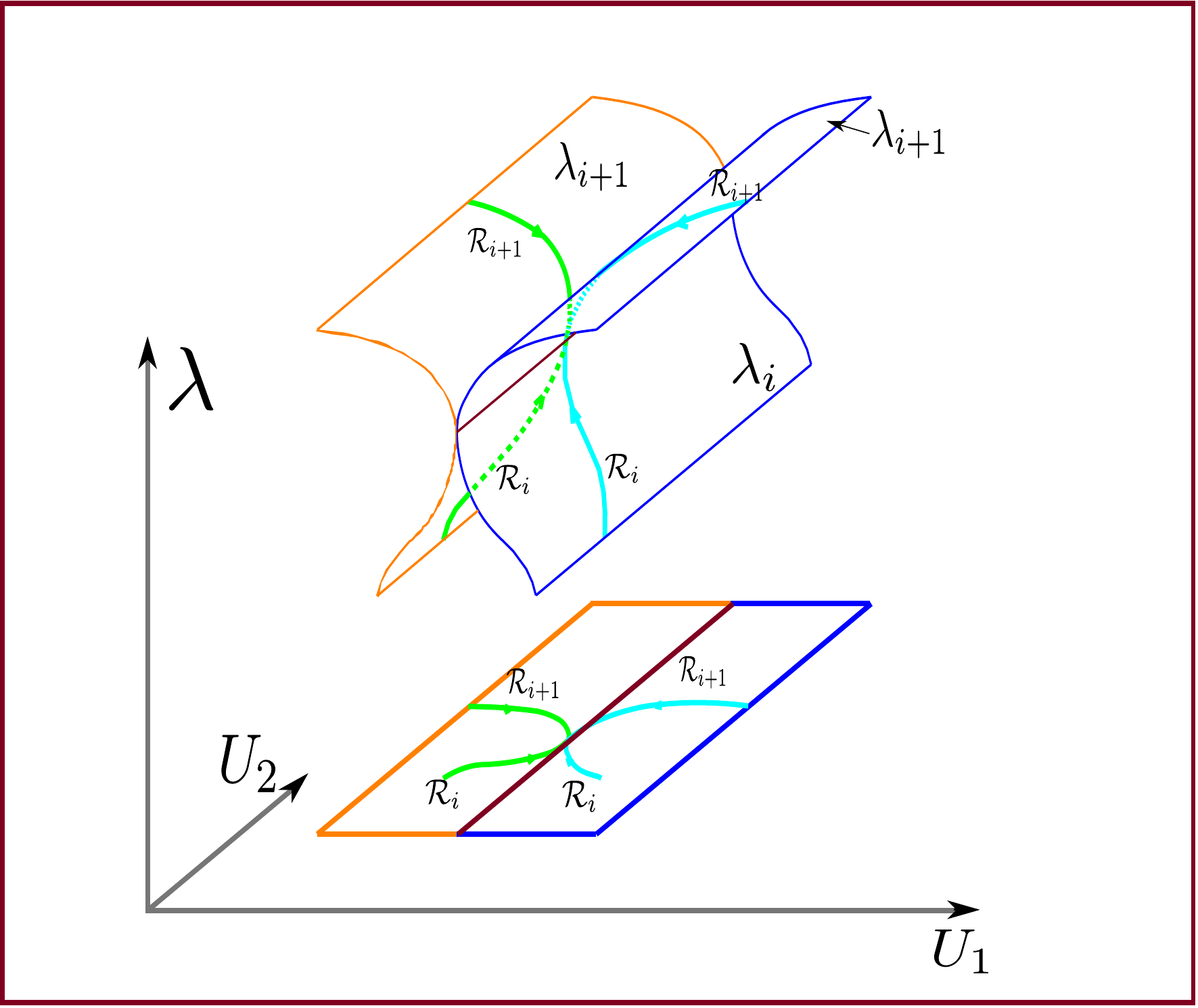}
	\caption{Three dimensional representation of new coordinate system for two dimensional case, as well as projections onto state space. $\mathcal{R}_i$ and $\mathcal{R}_{i+1}$ are the rarefaction curves of the $i$-th and $i+1$-th families.}
	\label{figFa}
\end{figure} 

\subsection{Generalized Jordan Chain}
\label{chain}

We assume that there exists $U_o$ such that the generalized eigenvalues of \eqref{ge1} satisfy $\lambda_o=\lambda_i(U_o)=\lambda_{i+1}(U_o)$ for some $i \in \{1,\ldots,n\}$. Let $A_o$ denote $A(U_o)$ for a non-singular matrix at $U_o$. We consider the case when $B_o \equiv B(U_o)$ is singular with linearly independent rows. In this case, the matrix $B_oB_o^T$ is invertible (since $rank(B_oB_o^T)=rank(B_o)$; see \cite{gentle2007matrix,ben2003generalized}), so that we can  define the right Moore-Penrose pseudoinverse of $B_o$ as $B_o^{\dagger}=B_o^T(B_oB_o^T)^{-1}$, for which $B_oB_o^{\dagger}=I_{n}$  (see \cite{elden1982weighted,hansen1998rank}), so that the matrix $B_o^{\dagger}$ is well defined on the subspace $R(A_o)$,   where $R(A_o)$ denotes
the range of matrix $A_o$.

Taking $M_o=B_o^{\dagger}A_o$, the generalized eigenvalue problem in \eqref{ge1} can be regularized as
 \begin{equation}
 M_or=\lambda r. \\
  \label{ge2}
 \end{equation}
We consider $M_o$ in the system \eqref{ge2} restricted to $N(B_o)^\perp$, where $N(B_o)$ is the kernel of $B_o$. Since the matrix $M_o$ forms a Jordan block
(double eigenvalue with a single eigenvector) then  there exists a single eigenvector $r_o \in N(B_o)^\perp$ associated to vector $r_1 \in N(B_o)^\perp$ determined by the \textit{Generalized
Jordan Chain} equations at the point $U_o$ (\cite{seyranian2003multiparameter}), i.e.
\begin{align}
M_or_o=\lambda_o r_o, \quad
M_or_1=\lambda_o r_1 + r_o.
\label{gen3}
\end{align}
Furthermore, there exists a left eigenvector $\bar{l}_o=l_oB_o$ and an associated left eigenvector $\bar{l}_1=l_1B_o$ such
that  
\begin{align}
\bar{l}_oM_o=\lambda_o \bar{l}_o, \quad
\bar{l}_1M_o=\lambda_o \bar{l}_1 +\bar{l}_o,
\label{gen4}
\end{align}
where the vectors $r_o,r_1,\bar{l}_1$ and $\bar{l}_o$ satisfy the relations
\begin{equation}
\bar{l}_or_o=0,\quad \bar{l}_or_1=\bar{l}_1r_o\neq 0, \quad \bar{l}_or_1=1, \quad \bar{l}_1r_o=1,\quad \bar{l}_1r_1= 0.
\end{equation}
%with the normalization condition 
%\begin{equation}
%\bar{l}_or_1=1, \quad \bar{l}_1r_o=1,\quad \bar{l}_1r_1= 0.
%\label{gen3f}
%\end{equation}
Equation \eqref{gen3} on space $N(B_o)^\perp$ can be rewritten as 
\begin{align}
B_o^{\dagger}A_or_o=\lambda_o r_o, \quad
B_o^{\dagger}A_or_1=\lambda_o r_1 +r_o,
\label{gen3a1}
\end{align}
or
\begin{align}
A_or_o=\lambda_o B_or_o, \quad
A_or_1=\lambda_o B_or_1 +B_or_o,
\label{gen3a1m}
\end{align}
using the right psedoinverse $B_o^{\dagger}$ ($B_oB_o^{\dagger}=I_n$).
To obtain $r_o$ and $r_1$ in system \eqref{gen3a1} we use the numerical method developed in \cite{mailybaev2006computation}.

\subsection{Local coordinate system at coincidence locus}
\label{chain1}

Rarefaction points form an
$n$-dimensional submanifold $\mathcal{C}$ of the fundamental manifold $\mathcal{W}$ called the characteristic manifold (see \cite{issacson1992global}). Associated to this submanifold there is a characteristic field defined as follows: any point $U \in \mathcal{C}$ corresponds to an eigenvector $r(U)\in T_U(\mathcal{C})$ of \eqref{ge1} with eigenvalue (characteristic speed) $\lambda$.  
When the system \eqref{eqt} is strictly hyperbolic, the manifold $\mathcal{C}$ is an n-sheeted
covering manifold for the state space. But in general systems the projections have singularities, for example when the generalized eigenvalues problem in \eqref{ge1} has multiple eigenvalues, i.e. coinciding characteristic speeds. Some ways to regularize
this singularity by means of new systems of coordinates can be found in \cite{palmeira1988line,issacson1992global,mailybaev2008hyperbolicity}.

The coincidence locus $\mathcal{E}$ constitutes an $n-1$-dimensional submanifold of the characteristic manifold $\mathcal{C}$ (see \cite{issacson1992global}). Therefore, points belonging to $\mathcal{E}$ present singularities with implications in the construction of admissible wave curves. Here, we regularize this singularity and find an asymptotic solution in a neighborhood of a point $U_o$ belonging to the coincidence locus. Our regularization method consist of a generalization of the one described in \cite{mailybaev2008hyperbolicity}.

We consider matrix $B$ in system \eqref{ge1} singular, while the case where $B$ is the identity matrix was solve in \cite{mailybaev2008hyperbolicity,mailybaev2008lax}. Regularization provides full access to all directions through a smooth field at coincidence points. These directions are required in order to construct every possible rarefaction curves. First, we take the smooth functions 
\begin{equation}
s(U)=(\lambda_i(U)+\lambda_{i+1}(U))/2-\lambda_o, \quad p(U)=(\lambda_i(U)-\lambda_{i+1}(U))^2/4,
\label{dos}
\end{equation}
which satisfy (see \cite{arnold2012geometrical,mailybaev2000transformation})
\begin{equation}
M(U)R(U)=R(U)N(U),  \quad 
\label{fin1}
\end{equation}
with
 \begin{equation}
 M(U)=B(U)^{\dagger}A(U),
 \end{equation}
 and
\begin{equation}
N(U)=\begin{bmatrix}
\lambda_o+s(U) & 1 \\
p(U) & \lambda_o+s(U)
\end{bmatrix}.
\end{equation}

Using that $B(U)B(U)^{\dagger}=I$, equation \eqref{fin1} can be rewritten as
\begin{equation}
A(U)R(U)=B(U)R(U)N(U),
\label{unoa1}
\end{equation}
where $R(U)=[R_o(U),R_1(U)]$ is a $m \times 2$ real matrix that depends smoothly
on $U$, while $s(U)$ and $p(U)$ are smooth real scalar functions such that
\begin{equation}
B(U_o)R_o(U_o)=B_or_o, \quad B(U_o)R_1(U_o)=B_or_1.
\end{equation}

Notice that $p \equiv 0$ defines the coincidence locus, which locally divides the space in two regions $\Omega_1$ and $\Omega_2$ where the system \eqref{eqt} is hyperbolic. 
The hyperplane tangent to the common boundaries of $\Omega_1$ and  $\Omega_2$ is given by
\begin{equation}
n_1 \cdot (U-U^*)=0, \quad \text{and} \quad n_2 \cdot (U-U^*)=0,
\end{equation}
where $\vec{n_1}=\triangledown p(U^*)$ and $\vec{n_2}=-\vec{n_1}$.

The characteristic surfaces corresponding to each eigenvalue family can be more easily represented and studied when lifted to the space $(U,\lambda)$ (e.g. Figures \ref{figFa} and \ref{figF}) . We can define a new system of coordinates onto these two surface branches. 

On branch $\Omega_1(U)$, which corresponds to $\lambda_i$ and where $n_1 \cdot (U-U^*) < 0$, we define 
	\begin{equation}
	\label{parameters}
	%\xi^2=(\lambda_i(U)-\lambda_{i+1}(U)),\quad \eta(U)=s(U),
	\xi^2=p(U),\quad \eta(U)=s(U), \quad \text{such that} \quad \lambda_{i}(U)=\lambda_o-\xi(U)+\eta(U).
	\end{equation}
\noindent Analogously, on branch $\Omega_2$ where $n_1 \cdot (U-U^*) > 0$, we have 
\begin{equation}
\lambda_{i+1}(U)=\lambda_o+\xi(U)+\eta(U).
\label{e22}
\end{equation}
It can be shown that put together these two eigenvalue sheets define a smooth surface in $(U, \lambda)$ space. Note that coincidence of the eigenvalues occurs for $\eta = 0$.

	The corresponding eigenvector $r$ associated to $\lambda_{i}$ and $\lambda_{i+1}$ are
	\begin{equation}
	\label{eige1}
	\begin{split}
	r(U) & =B(U)R(U)\begin{bmatrix}
	1 \\
	\pm \sqrt{(\lambda_i(U) - \lambda_{i+1}(U))}
	\end{bmatrix}\\
	& = B(U)R_o(U)\pm \sqrt{(\lambda_i(U) - \lambda_{i+1}(U))}B(U)R_1(U).
	\end{split}
	\end{equation}

The above transformation provides a natural coordinate system that can be used in order to complete the dimension up to $n$: $\xi$, $\eta$ and $n-2$ components of $U$, e.g. $(\xi, \eta, U_1, ... , U_{n-2})$. However, it is convenient to study the submersion of differentiable manifold given by the eigenvalues sheets into the $n+2$ dimensional space characterized by $(\xi, \eta, U)$, since it contains the smooth surface $\Omega$. In this space, \eqref{parameters} and \eqref{e22} provide a natural parametrization for the eigenvalues and eigenvectors. Moreover, rarefaction curves correspond to projections of integral curves on this surface onto state space $U$. The submersion is given by
\begin{equation}
\varphi(U,\lambda)=(\xi,\eta,U).
\end{equation}
In this new coordinate system we have that the surface in the lower part contains the rarefaction of the $i$-th family  while the upper part contains the rarefaction curve of the $(i + 1)$-th family (see Figure \ref{figFa}).

%\begin{eqnarray}
%\label{parametrized_eigenvectors2}
%r^a(U)= B(U)R_o(U)\pm \sqrt{(\lambda_i(U) - \lambda_{i+1}(U))}B(U)R_1(U).
%\label{eige1}
%\end{eqnarray}

 \begin{figure}
 	\centering
 	\includegraphics[width=0.79\textwidth]{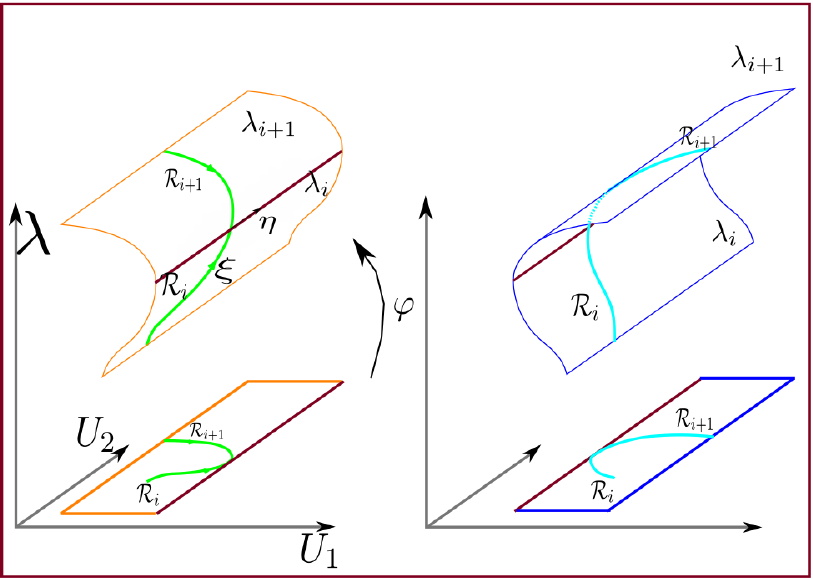}
 	\caption{Two dimensional representation of new coordinate system with separated branch surface. The charts $\varphi$ lifts the state space $(U,\lambda)$
 		into $(\xi,\eta,U)$ space.}
 	\label{figF}
 \end{figure}
 
 \subsection{System parameters in the regularized manifold}
 \label{chain21}
 
 In this section, we establish the relationship  between the regularized manifold and the system parameters in the neighborhood of the coincidence locus. We obtain the derivatives of $s(U)$, $p(U)$ and $R(U)$ in terms of the Jacobian and Hessian matrices of the accumulation $G$ and flux $F$ functions. Using these relations, we define a procedure for regularizing the eigenvector in \eqref{eige1}.
 
 %, which we use to leave the coincidence locus.
 
These derivatives are expressed in terms of $A_o$, $B_o$, $B_o^{\dagger}$ and the generalized Jordan chain $r_o$, $r_1$, $l_o$ and $l_1$ given in Section \ref{chain}. As proven in Appendix \ref{apend1}, the derivatives of $s$ and $p$ are given by (we leave it implied that all derivatives are evaluated at $U_o$)
	\begin{equation}
	\frac{\partial p}{\partial U_k}=l_o^T\frac{\partial A}{\partial U_k}r_o-\lambda_o l_o^T\frac{\partial B}{\partial U_k}r_o,
	\end{equation}
\begin{equation}
\frac{\partial s}{\partial U_k}=\frac{1}{2}\left(l_o^T\frac{\partial A}{\partial U_k}r_1+l_1^T \frac{\partial A}{\partial U_k}r_o\right)-\frac{\lambda_o}{2}\left(l_o^T\frac{\partial B}{\partial U_k}r_1+l_1^T \frac{\partial B}{\partial U_k}r_o\right).
\label{eq2a}
\end{equation} 
Let us define $Z=A_o-\lambda_o B_o+B_or_1l_1B_o$. The following holds (see Appendix \ref{apend1})
\begin{equation}
\frac{\partial R_o}{\partial U_k}=\frac{\partial s}{\partial U_k}r_1+\frac{\partial p}{\partial U_k}r_o+Z^{-1}\left(\lambda_o\frac{\partial B}{\partial U_k}r_o-\frac{\partial A}{\partial U_k}r_o\right),
\label{df4a}
\end{equation}
\begin{equation}
\frac{\partial R_1}{\partial U_k}=\frac{\partial s}{\partial U_k}r_o+Z^{-1}\left (\frac{\partial B}{\partial U_k}r_o+B_o\frac{\partial R_o}{\partial U_k}-\frac{\partial A}{\partial U_k}r_1+\lambda_o\frac{\partial B}{\partial U_k}r_1\right).
\label{df5b1}
\end{equation}
Using Taylor's formula to first order, $R_o(U)$ and $R_1(U)$ are approximated by

	\begin{equation}
		\begin{split}
	& R_o(U) = r_o+\sum_{k=1}^{n}\frac{\partial R_o}{\partial U_k}(U^k-U_o^k) + o(||U-U_o||^2), \\
	& R_1(U) = r_1+\sum_{k=1}^{n}\frac{\partial R_1}{\partial U_k}(U^k-U_o^k) + o(||U-U_o||^2).
	\end{split}
	\label{verssal}
	\end{equation}

where $U=(U^1,\dots,U^n)$ and $U_o=(U_{o}^1,\dots,U_{o}^n)$.

With the above parameter functions, it is possible to obtain asymptotic solutions of system \eqref{eqt} in the neighborhood of the coincidence locus. Using  $\vec{n_1}=\triangledown p(U^*)$, $\vec{n_2}=-\vec{n_1}$ and  equation \eqref{chanap} for $\partial p/\partial U_k$ in Appendix \ref{apend1}, we have a formula to calculate $\vec{n_1} \cdot r_o$ and $\vec{n_2} \cdot r_o$. When $\vec{n_1} \cdot p \neq 0$ and 
$\vec{n_2} \cdot p \neq 0$, we have the following asymptotic solutions $U_i$, $i=1,2$ on
$\Omega_1$ and $\Omega_2$ (derived in \cite{mailybaev2008hyperbolicity}):
\begin{equation}
U_1(\lambda)=U_o\pm\frac{(\lambda-\lambda_o)^2}{\vec{n_1} \cdot r_o} r_o+o((\lambda-\lambda_o)^2),
\label{dosp1}
\end{equation} 
and
\begin{equation}
U_2(\lambda)=U_o\pm\frac{(\lambda-\lambda_o)^2}{\vec{n_2} \cdot r_o}r_o+o((\lambda-\lambda_o)^2).
\label{dosp}
\end{equation} 

%esats expressoes da o valro da rarefaccao perto de Uo,
%para lam>lam0 e uma cosa em omega1
%para lam<lamo e outra coisa.

\section{Continuation beyond violation of strict hyperbolicity}
\label{proceM}

In this section, we describe an algorithm to continue rarefaction curves beyond the coincidence of characteristic speeds at point $U_o$, extending the capabilities of the wave curve method. A natural way of continuing a rarefaction starting at $U^- \in \Omega_1$ is to take
a close point $U_o^i=U_o+\epsilon r_o$, where $U_o^i \in \Omega_2$, $r_o$ is the eigenvector at $U_o$ and $\epsilon$ is a fixed small parameter. Then, we calculate the eigenvectors $r_i(U_o^i)$ and $r_{i+1}(U_o^i)$ associated with the eigenvalues $\lambda_i$ and $\lambda_{i+1}$ satisfying the condition $\triangledown \lambda_k \cdot r_k (U_o^i) >0$, with $k=i,i+1$ (see Figure \ref{comp1ab}). Afterwards, we continue the integration starting at a point $U_o^i$ by using the field associated to the family $i$ or $i+1$. In this way, at least two rarefaction curves can be constructed until some stopping criterion is satisfied.  We proceed in a similar fashion on the other side of the coincidence locus when the left state $U^-$ is situated on $\Omega_2$, i.e. for the family $i$ at a point  $U_o^i=U_o-\epsilon r_o$ for $U_o^i \in \Omega_1$. If the value of parameter $\epsilon$ is not small enough, error accumulation in the procedure can lead to a wrong right state $U_R$ in the Riemann solution.

The procedure described above is simple and it works in many cases. However when some displacement by $\epsilon$ in the direction of the generalized eigenvector $r_1=r_{i+1}(U_o^i)$ is taken (see Figure \ref{comp1ab}) the rarefaction sometimes still being almost tangent to the coincidence locus. Therefore, when the curve solver starts with initial direction almost parallel to the coincidence locus $\Gamma$, the integrator is not capable of capturing the true trajectory of the rarefaction curve. 
Another difficulty that appears with the choice of $\epsilon$ is that an inappropriate value may take $U_o^i$ took far from the true rarefaction trajectory and produce a vector for integration of the wave curve in the opposite direction.

These numerical difficulties appear due to resonance phenomenon, which stands for a coincidence of characteristics speeds along a manifold or isolated points in the literature of Riemann problems. In the neighborhood these states, the behavior of waves is strongly sensitive to the curvature of the coincidence manifold. The construction of solutions involving resonant waves requires careful numerical analysis in order to develop a robust algorithm.

One of the reasons for this behavior is that the parameter $\epsilon$ does not contain information about of flux and accumulation and therefore does not represent an indicator of variation of the solution at each point.
\begin{figure}
	\begin{center}
		\includegraphics[scale=0.50]{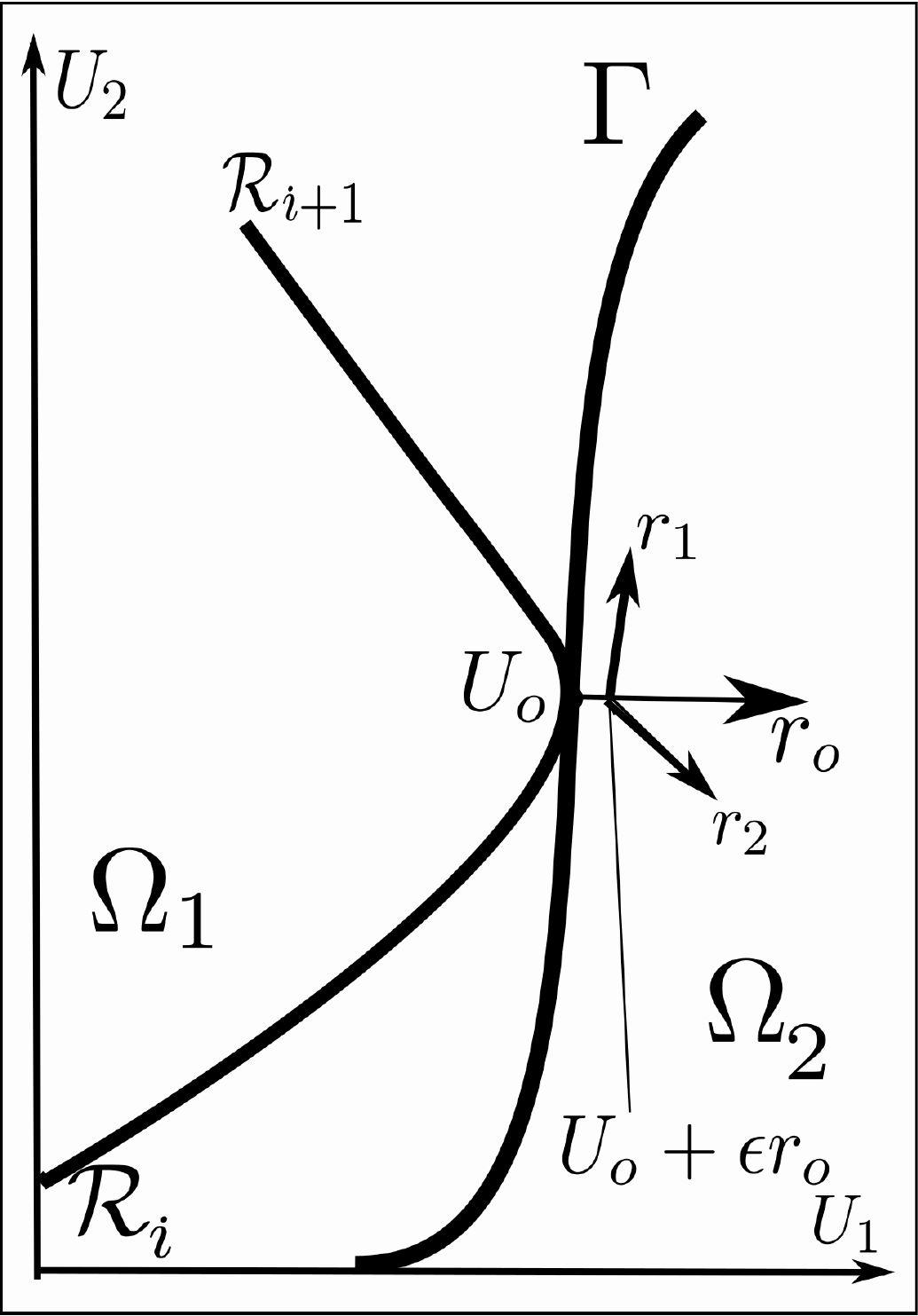}
		\caption{Here $\Gamma$ represents the coincidence between eigenvalues $\lambda_i$ and $\lambda_{i+1}$, while $r_1$ and $r_2$ are the eigenvectors of the generalized problem \eqref{ge1} at point $U_o+\epsilon r_o$. }
		\label{comp1ab}
	\end{center}
\end{figure}
Improving upon this numerical difficulty, we use the coordinate system described in Sections \ref{chain1} and \ref{chain21} as well as the relations derived there.  We choose the value of parameter $\epsilon$ in order to vary the eigenvalue in a way that the new eigenvector approximates the tangent to the rarefaction curve accurately; one the considerations involved in the choice of $\epsilon$ is to take into account the local curvature.

We present the following \textbf{Algorithm $3.1$} to provide a continuation of the rarefaction beyond the coincidence locus by assuming that $\triangledown \lambda_i \cdot r_o > 0$. Moreover, the vectors $R_o$ and $R_1$ in \eqref{verssal} are taken such that $R_o \cdot r_o > 0$ and  $R_1 \cdot r_o > 0$. The procedure below is for $U_L \in \Omega_1$, meaning that \gui{the rarefaction originates from $\Omega_1$ and continue inside region $\Omega_2$}.\\

%either at inflection or coincidence points or in the physical boundary
\gui{\noindent{\textbf{Algorithm 3.1:}}}
\begin{itemize}
	\item[1)] Construct a rarefaction curve of family $i$ solving ODE \eqref{RH4} starting at $U_L \in \Omega_1$ and stopping at the coincidence point $U_{o}$;
	
	\item[2)] Use formula \eqref{dosp} to calculate $U_o^2=U_2(\lambda)$ in $\Omega_2$, with $\lambda-\lambda_o=\epsilon$ for a fixed $\epsilon$. 
	
	\item[3a)]Compute the eigenvector $r_i (U_o^2)$, taking a direction such that $\triangledown \lambda_i \cdot r_i (U_o^2) >0$. If $r_i (U_o^2) \cdot r_o > 0$ then continue the rarefaction wave by solving ODE $d U/ d\xi=r_i(U)$ with $U(0)=U_o^2$ until some stopping criterion is satisfied;
	
	\item[3b)] Compute the eigenvector $r_{i+1} (U_o^2)$ and $R_1(U_o^2)$ (see formula in \eqref{verssal}) with directions satisfying $\triangledown \lambda_{i+1} \cdot r_{i+1} (U_o^2) >0$. If $r_{i+1} (U_o^2) \cdot r_o > 0$ and $R_1(U_o^2) \cdot r_{i+1} (U_o^2) >0 $ then continue the wave curve with the solution of $d U/ d\xi=r_{i+1}(U)$ with $U(0)=U_o^2$ and redefined $r_{i+1}(U_o^2)=R_1(U_o^2)$, until some stopping criterion is satisfied;
	
\end{itemize}	

\textbf{Algorithm $3.1$}  can be similarly adapted for wave curves crossing from $\Omega_2$ to $\Omega_1$ (when $U_L \in \Omega_2$) by using formula \eqref{dosp} to calculate $U_o^1=U_1(\lambda)$ as following

\begin{itemize}
	
	\item[4)] Construct a rarefaction curve of family $i$ solving ODE \eqref{RH4} starting at $U_L \in \Omega_2$ and stopping at the coincidence point $U_{o}$;
	
	\item[4a)] Use the formula \eqref{dosp1} to calculate $U_o^1=U_1(\lambda)$ in $\Omega_1$, with $\lambda-\lambda_o=\epsilon$ for a fixed $\epsilon$;
	
	\item[4b)]Compute the eigenvector $r_i (U_o^1)$, taking a direction such that $\triangledown \lambda_i \cdot r_i (U_o^1) >0$. If $r_i (U_o^1) \cdot r_o > 0$ then continue the rarefaction	wave by solving ODE $d U/ d\xi=r_i(U)$ with $U(0)=U_o^1$ until some stopping criterion is satisfied;

	\item[4c)]Compute eigenvector $r_{i+1} (U_o^1)$, taking directions such that $\triangledown \lambda_{i+1} \cdot r_{i+1} (U_o^1)$ $>0$. If $r_{i+1} (U_o^1) \cdot r_o > 0$ and $R_1(U_o^1) \cdot r_{i+1} (U_o^1) >0 $ (region $\Omega_1$) then continue the rarefaction curve by solving ODE $d U/ d\xi=r_i(U)$ with $U(0)=U_o^1$ until some stopping criterion is satisfied.
	
\end{itemize}

Since the vector $r$ is an eigenvector, so is $-r$; we must be careful in choosing the correct direction for the eigenvector in order for the procedure to work.
For this reason, we take the Jordan chain vector $r_1$ with direction such that  $\lambda_i^\prime (r_1) > 0$ in coincidence locus.

\textbf{Algorithm $3.1$} improves the continuation of rarefaction beyond the coincidence locus and expresses a theoretical argument to the existence of \gui{wave curves} after a coincidence locus. The major challenge left is the choice of an appropriate value for $\epsilon$, which must be made by the user while taking into account the parameters of the model. From a numerical point of view, the choice of $\epsilon$ as $\lambda-\lambda_o$ together with the choice of $U_o^1$ as an asymptotic solution of the rarefaction curve close to coincidence locus, guarantees that we obtain a new point $U_o^1$ more accurately avoiding some of the above mentioned difficulties.

The vectors $R_o$ and $R_1$ in formula \eqref{verssal} are obtained approximately from the first order truncation of the Taylor series and the eigenvector $r_i$ at a point $U_o^1$ close to coincidence point. This approximation of first order can be inaccurate if users supply an inappropriate choice for parameter $\epsilon$, but considering higher order terms can be numerically impractical. 

%We have numerical evidence that an appropriate starting field of the rarefaction at point $U_o^1$ is efficiently obtained by taking a combination of $R_1$ and $r_i(U_o^1)$.
	
\subsection{Rarefaction followed by rarefaction of another family}

A sequence of rarefaction and shock waves is represented in state space as a concatenation of wave curves. From here on, we use the notation $\mathcal{R} \rightarrow \mathcal{S}$ in order to say that a shock curve $\mathcal{S}$ is concatenated after a rarefaction wave $\mathcal{R}$.

A rarefaction curve can be followed by a rarefaction from another family if strict hyperbolicity is lost. When the rarefaction curves $\mathcal{R}_i$ and $\mathcal{R}_{i+1}$ cross the surface $\Gamma$ and meet at the point $U_o$, Algorithm $3.1$ reproduces the following possibilities:

\begin{itemize}
	\item When $\lambda_i^\prime (r_o) \neq 0$, $\mathcal{R}_i \in \Omega_1$ is continued beyond $U_o$ in one of the following sequences:
	
	a) if $\lambda_i^\prime (r_o) > 0$ in $\Omega_2$, then $\mathcal{R}_i \rightarrow \ \mathcal{R}_i$, with $\mathcal{R}_i \in \Omega_2$;
	
	b) if $\lambda_{i+1}^\prime (r_o) > 0$ in $\Omega_2$, then $\mathcal{R}_i \rightarrow \ \mathcal{R}_{i+1}$, with $\mathcal{R}_{i+1} \in \Omega_2$;
	
	c) if $\lambda_{i+1}^\prime (r_o) > 0$ in $\Omega_1$, then $\mathcal{R}_i \rightarrow \ \mathcal{R}_{i+1}$, with $\mathcal{R}_{i+1} \in \Omega_1$;
 
 \item When $\lambda_i^\prime (r_o) = 0$, in order to continue $\mathcal{R}_i \in \Omega_1$ beyond $U_o$:

d) if $\lambda_{i+1}^\prime (r_o) > 0$ in $\Omega_2$, then $\mathcal{R}_i \rightarrow \ \mathcal{R}_{i+1}$ for $\mathcal{R}_{i+1} \in \Omega_2$;

e) if $\lambda_{i+1}^\prime (r_o) > 0$ in $\Omega_1$, then $\mathcal{R}_i \rightarrow \ \mathcal{R}_{i+1}$ for $\mathcal{R}_{i+1} \in \Omega_1$.

\end{itemize} 

\begin{figure}
	\begin{center}
		\includegraphics[scale=0.39]{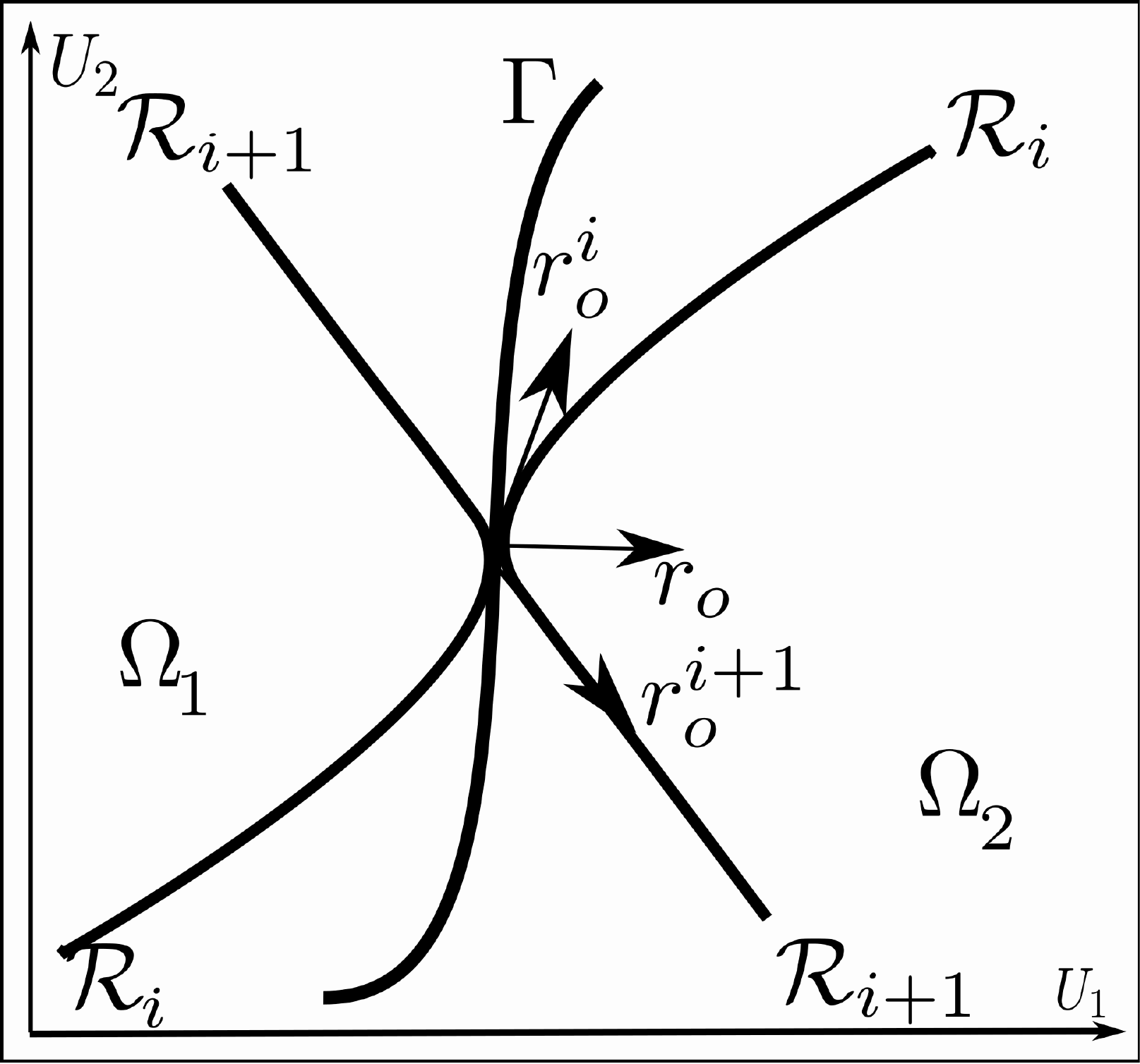}
		\caption{Rarefaction curves for two dimensional case. The arrows indicate the growth of the eigenvalues. Here $r_o^i$ and $r_o^{i+1}$ are the projections of $R_o$ and $R_1$ on phase space $U$. }
		\label{comp1a}
	\end{center}
\end{figure}

\subsection{Composites: Rarefactions followed by shocks} 
It is also possible to continue a rarefaction with shock curve, forming a pair of concatenated rarefaction and shock curves.
 
A shock wave curve for a fixed left state $U^-$ is formed by the set of right states 
\begin{equation}
\mathbb{H}(U^-)=\{U^+ : F (U^-) - F (U^+) - s(G(U^-) - G(U^+))=0\},
\label{hl1}
\end{equation}
where $s$ is the shock velocity (see details in \cite{matos2015bifurcation}).

%We study this waves in details in Sections \ref{secomp}.
We have the following types of continuation for $\mathcal{R}_i$ beyond the coincidence point $U_o$:

(1)  $\mathcal{R}_i \rightarrow \ \mathcal{S}_i$, with $\mathcal{R}_i \in \Omega_2$;

(2) $\mathcal{R}_i \rightarrow \ \mathcal{S}_{i+1}$, with $\mathcal{S}_{i+1} \in \Omega_2$;

(3) $\mathcal{R}_i \rightarrow \ \mathcal{S}_{i+1}$, with $\mathcal{R}_{i+1} \in \Omega_1$.

There is another situation where the continuation can be a contact. The procedure to construct the pairs of wave curves is similar.

When either $\lambda_i^\prime (r_o) =0$ or $\lambda_{i+1}^\prime (r_o) =0$ occurs, the corresponding rarefaction stops at the point $U_o$. The continuation in this case require of the composite wave curve, which is studied in details in Section \ref{comp}.

\amf{
	\subsection{Numerics for the construction of Hugoniot curve}
	
	In the Riemann solver, we find a Hugoniot curve using two methods. The first is a direct search for points satisfying condition \eqref{hl1}, by means of sign changes of the functions in an appropriate grid. The second consists of the continuation method, which is useful in the construction of wave curve. For this reason, we describe here the general idea of the continuation algorithm implemented in our solver.
	
	Let us denote by $H(U,s)=F (U^-) - F (U) - s(G(U^-) - G(U))$ the function defined on $\mathbf{R}^{n+1}$ which takes values in $\mathbf{R}^{n}$. We want to compute a curve satisfying $H(x)=0$,
	with $x=(U,s)$. Starting by the point $U^o=U^-$, we find a new point of curve as follows
	\begin{equation}
	U^1=U^o+hr_o,
	\end{equation}
	where $h$ is certain stepsize and $r_o$ is normalized tangent vector at $U^o$, i.e., $(\partial H/\partial U)r_o=0$. The new point on the curve is obtained by a Newton-like procedure to search the intersection point of equation $H(x)=0$ with the plane $g(x)=0$, where
	$g(x)=(x-U^1) \cdot r_o$. In each step we update $U^o$ by $U^1$ and $r_o$ by the normalized eigenvector at $U^1$. The iterative process is repeated until convergence.  
}

\subsection{Condition for stopping at coincidence locus}
\label{coin1}

The first step in \textbf{Algorithm $3.1$} is to identify when a rarefaction curve reaches coincidence locus. This is easy when analytical formulas for the eigenvalues $\lambda_i$ and $\lambda_{i+1}$ are available. Unfortunately, this situation does not correspond to the general setting, where we usually have approximations of eigenvalues computed automatically. Our strategy for these cases is to construct a continuous function that has different signs on each side of the coincidence locus. Then, this sign changes serves as the stopping criterion for the ODE solver.

In this section we propose an algorithm to detect these changes for a 2D hyperbolic system of equations, which covers most of our applications. We suppose that explicit formulas for $F$, $G$ and their Jacobian matrices are given. Since matrix $B$ might be singular, we consider the equivalent problem of eigenvalues for matrix $M = B^\dagger A$ (\amf{for $A$, $B$ defined in \eqref{ge1} and $B^\dagger$ defined in Section \ref{chain}}). \amf{Again, let $\lambda_1 < \lambda_2$ denote the eigenvalues of \eqref{ge1}, i.e. roots of the characteristic polynomial $p(\lambda)=det(A-\lambda B)$, which can be found explicitly as},
	\begin{equation}
	\lambda_{1,2}=(trM \pm \sqrt{D})/2, \quad \text{where} \quad D = (trM)^2 -4detM.
	\label{discr}
	\end{equation}

Note that the discriminant $D$ is zero in a coincidence locus, but it is always non negative since we are restricted to hyperbolic systems and cannot be easily used to define a stopping condition for the solver. However, various methods are available to find a set of zeros of the discriminant in state space, thus characterizing the coincidence loci. After that, we propose \textbf{Algorithm 3.2} to identify when a wave curve intersects a coincidence locus, in order to stop the regular curve integrator and initiate the continuation procedure with  \textbf{Algorithm 3.1}:	\\

%.
\gui{\noindent{\textbf{Algorithm 3.2:}}}
\begin{itemize}
	\item[1)] Construct a rarefaction curve of family $i$ solving ODE \eqref{RH4} starting at $U_L \in \Omega_1$, at each step of the solver;
	
	\item[2)] During the integration procedure, for a given tolerance $\epsilon$ calculate the point $U_r^{\epsilon}$ such that discriminant $D$ in \eqref{discr} satisfies $|D(U_r^{\epsilon})| \le \epsilon$;
	
	\item[3)] Using $U_r^{\epsilon}$ as starting point, use procedure in \cite{mailybaev2006computation} and calculate the coincidence point $U^+_{o}$ where
	$\lambda_1(U^+_{o})=\lambda_2(U^+_{o})$;
	
	\item[4)] Define the distance $d_r=d(U_r^{\epsilon},U^+_{o})$ between point $U_r^{\epsilon}$ and $U^+_{o}$. Define the function $f(U_i)=d_r-d(U_r^{\epsilon},U_i)$ where $U_i$ denotes the state along the rarefaction curve during the integration procedure;
	
	\item[5)] As stopping criterion use the fact that $f(U_i) >0$ until $U^+_{o}$ and $f(U_i) < 0$ before the point $U_i$ crossing
	the coincidence locus.
	
\end{itemize}
Step 3 of \textbf{Algorithm 3.2} requires  the method developed in \cite{mailybaev2006computation} in order to accurately determinate the intersection point of two eigenvalues from a known close point. This method combines the versal deformation theory \cite{mailybaev2001transformation} and the Schur canonical form \cite{golub1996matrix} to characterize the locus where eigenvalues coincide. This procedure was implemented in a routine that computes multiple eigenvalues and generalized eigenvectors for matrices dependent on parameters.

%These tools are used in the following Algorithm $3.2$ to calculate the point where a rarefaction curve meets a coincidence locus. At this point, the traditional continuation procedure for rarefaction curves should be stopped and an implementation of Algorithm $3.1$ should be used. 

Notice that the stopping criterion in Step 5 of \textbf{Algorithm 3.2} must introduce in the ODE solver as an event function which is calculate in each step of the integration procedure. This process is based on event localization during integration, which consists of characterizing the sought event (in our case the intersection with a coincidence locus) as a zero of a continuous function. After each step of the integration, it is checked if this function changes sign to invoke refinement of the solution and localization of the event.
Several paper addressed the problem of detection and location of events, see e.g. \cite{newman1968numerical,moler1997we,shampine1997matlab,shampine2000event} and cited there in. In our Riemann solver we use the ODE solver of MATLAB, providing the event associated to the function $f$ of step four of \textbf{Algorithm 3.2} that changes it sign when crossing the coincidence locus.

\subsection{Stopping at a planar boundary}
\label{plane}

Here we suggest an alternative procedure for Step 2 of \textbf{Algorithm 3.2} when the coincidence curve is known to be planar and which identifies intersection points with great precision. This is a very useful procedure, since it may be efficiently implemented to deal with various types of stopping conditions along curves, e.g. domain boundaries, coincidence and inflection loci. Moreover, this algorithm can be used in a large array of applications in other areas.

We are interested in identifying when an orbit crosses $n$-dimensional hyperplanes $P=\{x: a_1 x_1 + ... + a_n x_n=d\}$ in order to stop the continuation or change to the appropriate procedure. Let us define $z = \vec{a} \cdot x-d$, with $ \vec{a}=(a_1,\dots,a_n)$. Differentiating the variable $z$ with respect to $x_1$, we obtain
	
\begin{eqnarray}
	\frac{dz}{dx_1} = a_i + \sum_{j=1; j \neq i}^n a_j \frac{dx_j}{dx_1}.
	\label{dfg}
\end{eqnarray}
	
System \eqref{eq:rarefaction_curve}, which describes the rarefaction curves, can be explicitly written for any chosen family of index $j$ as
	\begin{equation}
	\label{eq:planar_boundary_1}
	\frac{dx_j}{d\xi} = r_j(x),
\end{equation}
\noindent for initial data $x_{j}(\xi=0) = x_0$. Supposing that $r_1 \neq 0$, we obtain $dx_j / dx_1 = r_i / r_1$ from \eqref{dfg} and \eqref{eq:planar_boundary_1}. This leads to equation
\begin{equation}
	\label{eq:planar_boundary_2}
	\frac{dz}{dx_1} = a_1 + \sum_{i=2}^n a_i \frac{r_i}{r_1},
\end{equation} 
with initial data $z(x_0)$, which is integrated together with to system \eqref{eq:planar_boundary_1}. We use an integrator for this system} with a fixed step size $h$. At each step $i$ we calculate the distance $d_i$ between the newly generated point $(x_i,z_i)$ and the plane $P$. Then, event detection algorithms can be used to stop the integration whenever $d_i < h$. The last step is done integrating with step size $d_i$ until $z-d$ reaches zero.
	
\subsection{Condition for stopping at inflection locus}
\label{subsect:stop_inflection}
We define the inflection locus for the family $k$ as the set of states where $\triangledown \lambda_{k} \cdot r_{k}=0$. This condition can be expressed in more detail:
 \begin{equation}
 \triangledown \lambda_{k} \cdot r_{k}
 =l_k^T \cdot \left(r_k^T\frac{\partial^2F}{\partial W^2}r_k-\lambda _k r_k^T\frac{\partial^2G}{\partial W^2}r_k\right)/(l_k^T \cdot Br_k)
 =0 ,
 \label{infleL}
 \end{equation}
 where $B$ is the Jacobian of the accumulation $G$ and $\frac{\partial^2F}{\partial W^2}$, $\frac{\partial^2G}{\partial W^2}$ denote the Hessian of accumulation and flux, respectively. We denote by $r_k$ and $l_k$ the right and left generalized eigenvector of the matrix in \eqref{ge1}, respectively (see deduction of \eqref{infleL} in \cite{helmut2011thermal}).

Formula \eqref{infleL} is valid if there are right and left eigenvectors $r_k$ and $l_k$ at each point such that $l_k^T \cdot Br_k=0$ . But this not true on the coincidence locus
where the matrix $M=B^\dagger A$ takes the form of a Jordan block. Therefore, we use formula 
 \eqref{infleL} as a stop criterion for an inflection locus that does not coincide with a coincidence locus. When this intersection happens, it suffices to employ \textbf{Algorithm $3.1$} and the criterion described in Section \ref{coin1}.

\subsection{Application to the ICDOW model}

In this section we exemplify the appearance of coincidence locus with the model studied 
in \cite{alvarez2017analytical}. The procedure developed in this paper is used to solve a particular Riemann problem.

%In this section, we  take as example the model studied in \cite{alvarez2017analytical} where the coincidence locus appears. Solving a particular Riemann problem is used the procedure derived in this paper.

We take the system of three conservation laws disregarding 
diffusive terms 
 \begin{align}
 \partial_{t}\left(  \varphi  s_{w} \grdel{\rho_{1}}  \right) 
 +\partial_{x}\left(
 uf_{w}\grdel{\rho_{1}} \right)  =  0,  %
 \label{eq:balance1a}\\
 \partial_{t}\left(  \varphi ( s_{w} \grdel{\rho_{2}}   + s_{o}\grdel{\rho_{3}}) \right) 
 +\partial_{x}\left(
 u(f_{w} \grdel{\rho_{2}} + f_{o}%
 \grdel{\rho_{3}})\right)  =  0,  %
 \label{eq:balance3a}\\
 \partial_{t}\left(  \varphi s_o \grdel{\rho_{4}} \right) 
 +\partial_{x}\left(
 u f_o \grdel{\rho_{4}}%
 \right)  =  0,  %
 \label{eq:balance4a}
 \end{align}
for the unknowns water saturation $s_w$, $y$ and the Darcy velocity $u$, i.e. $U=(s_w,y,u)$. We also have $s_o=1-s_w$ and $f_o=1-f_w$. The parameter $\varphi$ denotes the porosity of the media. The molar density functions \grdel{$\rho_{1}$, $\rho_{2}$, $\rho_{3}$ and $\rho_{4}$} are positive and differentiable functions that depend only on the variable $\grdel{y}$.
 
 The system of conservation laws (\ref{eq:balance1a})-(\ref{eq:balance4a}) can be rewritten as:
 \begin{align}
 &\frac{\partial G(s_w,y)}{\partial t}+\frac{\partial ( u \widehat{F}(s_w,y))}{\partial x}=0,\label{leicons2a1}
 \end{align}
 where the accumulation and flux functions are written as
 	\begin{align}
 	\label{fsombre1}
 	(G_1,G_2,G_3)^T&=\varphi( s_{w}\rho_{1}, s_{w}\rho_{2}  + s_{o}\rho_{3}, s_o\rho_{4})^T,
 	\\
 	(\widehat{F}_1,\widehat{F}_2,\widehat{F}_3)^T 
 	&= (f_{w}\rho_{1},f_{w}\rho_{2}+ f_{o}%
 	\rho_{3},f_o \rho_{4})^T.
 	\label{fsombre}
 	\end{align}
We are interested in the Riemann-Goursat problem  associated with $(\ref{leicons2a1})$,  
 i.e. the solution of $(\ref{leicons2a1})$ with piecewise constant initial and boundary data
 \begin{equation}
 \left\{
 \begin{array}
 [c]{ll}%
 (s_{wl},y_l,u_l) &
 \text{if} \quad x=0, t > 0, \\
 (s_{wr},y_r,u_r) & \text{if} \quad x > 0, t=0.
 \end{array}
 \right.\label{Riemancondition}
 \end{equation}
 The value of $u_r$ on the right state is obtained from the model. Notice that the accumulation function $G$ does not depend on the variable $u$. Thus, the Jacobian of $\partial G/\partial U$ has null third column. Moreover, this variable appears in the flux function multiplying a function that depends on $(s_w,y)$.
  This case perfectly falls in the class of problems studied in Section \ref{chain}, thus we can use
  the procedure developed in this work.
  
  After solving the generalized eigenproblem \eqref{ge1} for the ICDOW model (see \cite{alvarez2017analytical} for details) we obtain the eigenvalues 
   \begin{equation}
   \lambda_s=\dfrac{u}{\varphi}\frac{\partial f_w}{\partial s_w}, \quad \text{and} \quad 
   \lambda_H=\dfrac{u}{\varphi}\dfrac{\Delta_1 f_w+ \Delta_2}{\Delta_1 s_w+ \Delta_2},
   \label{lams}
   \end{equation}
  where $\Delta_i$ depend only on $y$, with $i=1,2$.
  
  \begin{figure}[h]
  \begin{center}
  		\includegraphics[scale=0.43]{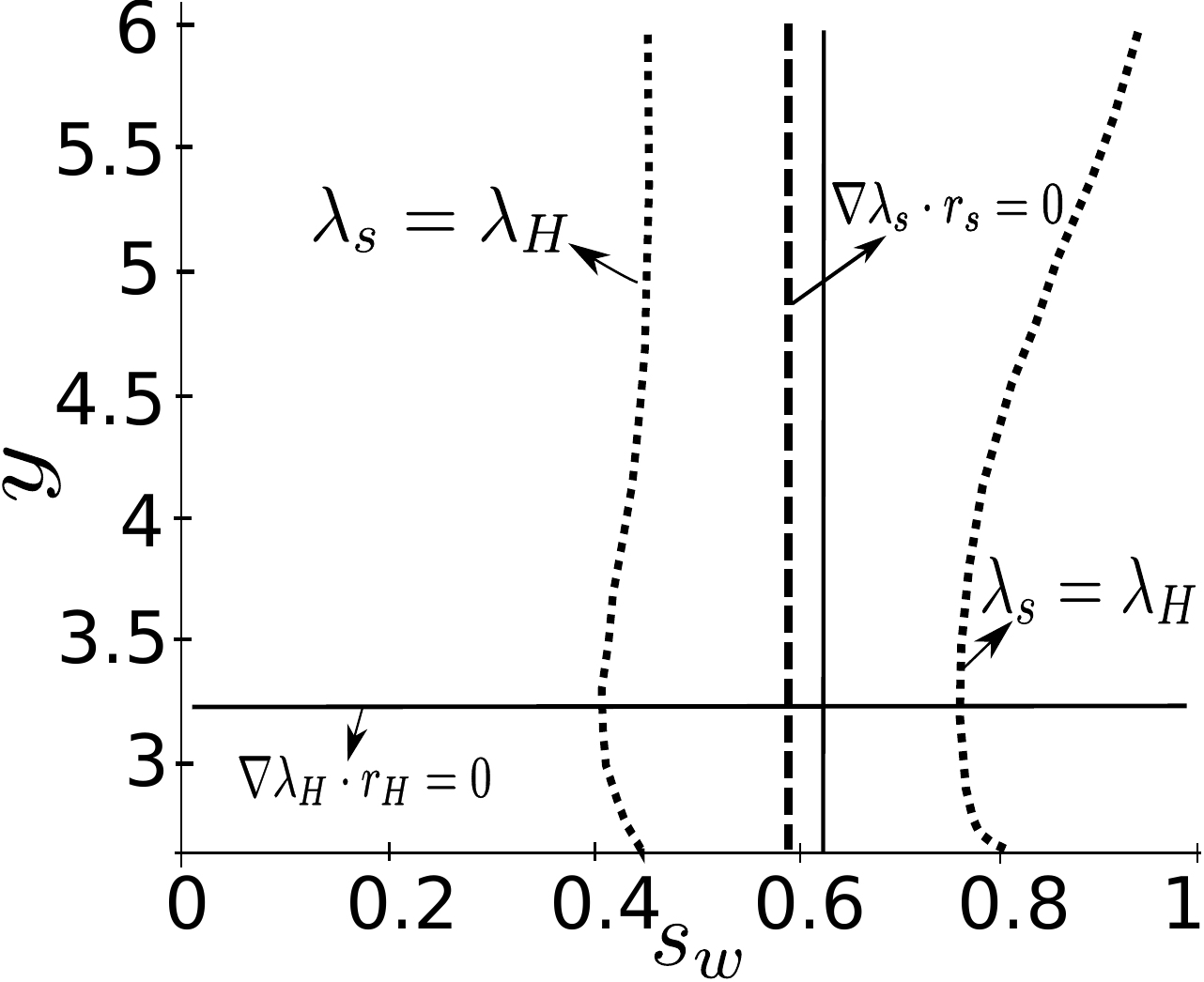}
  \end{center}
  	\caption{Coincidence and inflection loci in the projected phase state $(s_w,y)$. 
  	 Dashed line represents 
  			the inflection locus for the saturation wave $\triangledown \lambda_{s} \cdot r_{s}=0$, dot lines represent the coincidence $\lambda_H=\lambda_{s}$ and the bold line represents the inflection locus for the case of composition wave curves $\triangledown \lambda_H \cdot r_H=0$. We have that $\{\lambda_H=\lambda_{s}\} \subset \{ \triangledown \lambda_H \cdot r_H=0 \}$. }
  		\label{phasespacea1}
  \end{figure}	
  The eigenvector of characteristic system  $\partial F/\partial U-\lambda \partial G/\partial U=0$  for $\lambda_s$ is $\vec{r}_{s}=(1,0,0)^T$ (saturation wave)
   while the  eigenvector for $\lambda_H$ is $\vec{r}_H=(r_H^1,r_H^2,r_H^3)$, which is called the chemical composition wave because mainly the chemical variable $y$ changes.
   In this case, we have two families of rarefaction curves, the saturation rarefaction, denoted by $\mathcal{R}_{s}$ and
  	the chemical rarefaction, denoted by $\mathcal{R}_H$. These curves are obtained as integral curves of each eigenpair, i.e.    
  	$d \mathcal{R}_{s}/d \xi=\vec{r}_{s}$ 
  	and $d \mathcal{R}_H/d \xi=\vec{r}_H$. The other curves necessary for constructing 
  	the Riemann solution are the shock curves which represent the discontinuous solution of 
  	 \eqref{eq:balance1a}-\eqref{eq:balance4a}. In  phase space $(s_w,y,u)$ these discontinuous solutions form the Rankine-Hugoniot locus (RH-locus).  For a  given left state $U^-=(s_w^-, y^-, u^-)$,   the  RH-locus is the set of right states $U^+=(s_w^+, y^+, u^+)$ that satisfy the  Rankine-Hugoniot relationships
  	\begin{equation}
  	\sigma (G_i(s_w^+, y^+)-G_i(s_w^-, y^-))=u^+\widehat{F_{i}}(s_w^+, y^+)-u^-\widehat{F_{i}}(s_w^-, y^-),
  	\label{rh1a}
  	\end{equation}
  	for $i=1,2,3$. Here $\widehat{F_{1}}$, 
  	$\widehat{F_{2}}$ and $\widehat{F_{3}}$ are given by (\ref{fsombre}) while $G_1$, 
  	$G_2$ and $G_3$ are given by (\ref{fsombre1}). The function $\sigma=\sigma(U^-,U^+)$ represents the shock speed between the states $U^-$ and $U^+$. We denote by $S_s$ the shock associated with the saturation wave and $S_H$ associated with the chemical wave.
  	
  In Figure \ref{phasespacea1} we present the bifurcation curves, i.e. inflection and coincidence loci, for ICDOW model with coefficient function $\rho_i$ appearing in
  \cite{alvarez2017analytical}. We use the fractional function $f_w$ show in Figure \ref{fractional} left. A particularity of the bifurcation curves in this case is that 
  the coincidence locus belongs to the inflection locus of the chemical family, i.e. 
  $\{U:\lambda_H(U)=\lambda_{s}(U)\} \subset \{U:\triangledown \lambda_H \cdot r_H (U)=0\}$. 
 
\begin{figure}[h]
	\begin{center}
		\includegraphics[scale=0.39]{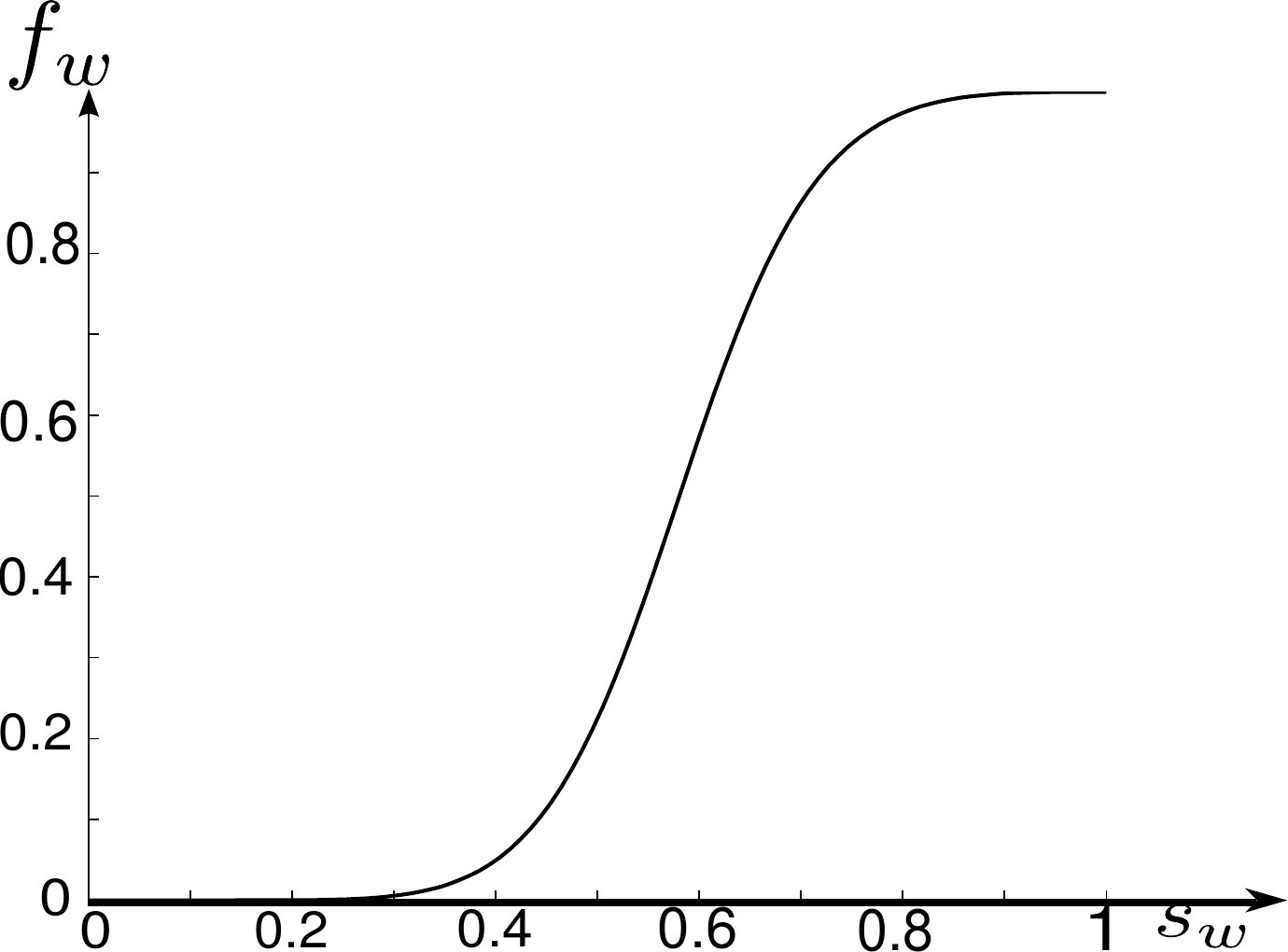}
		\includegraphics[scale=0.42]{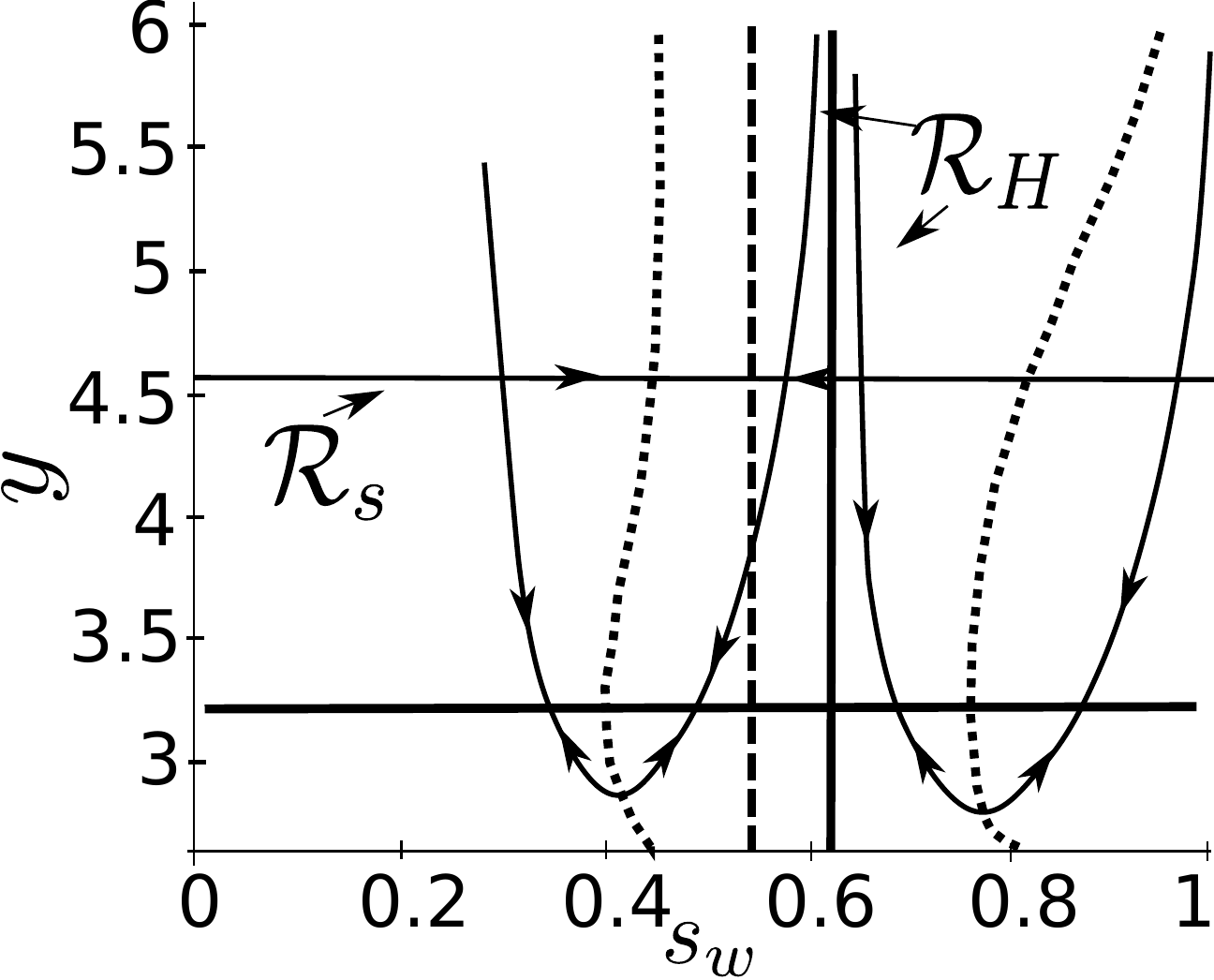}
	\end{center}
	\caption{\textit{a - left)} Fractional flow $f_w$ given by Eqs.   \eqref{eq:balance4a}-\eqref{eq:balance4a}.
		\textit{b - left)} Rarefaction $\mathcal{R}_s$ and $\mathcal{R}_H$ the arrows indicates the direction of increasing of $\lambda_s$ and $\lambda_H$. Shock $\mathcal{S}_s$ is also a straight line parallel to axis $y$. Shock $\mathcal{S}_H$ has a similar form that $\mathcal{R}_H$ for states close to coincidence of the eigenvalues curve $\mathcal{C}$. }
	\label{fractional}
\end{figure}
 \begin{figure}[h]
 	\begin{center}
 		\includegraphics[scale=0.42]{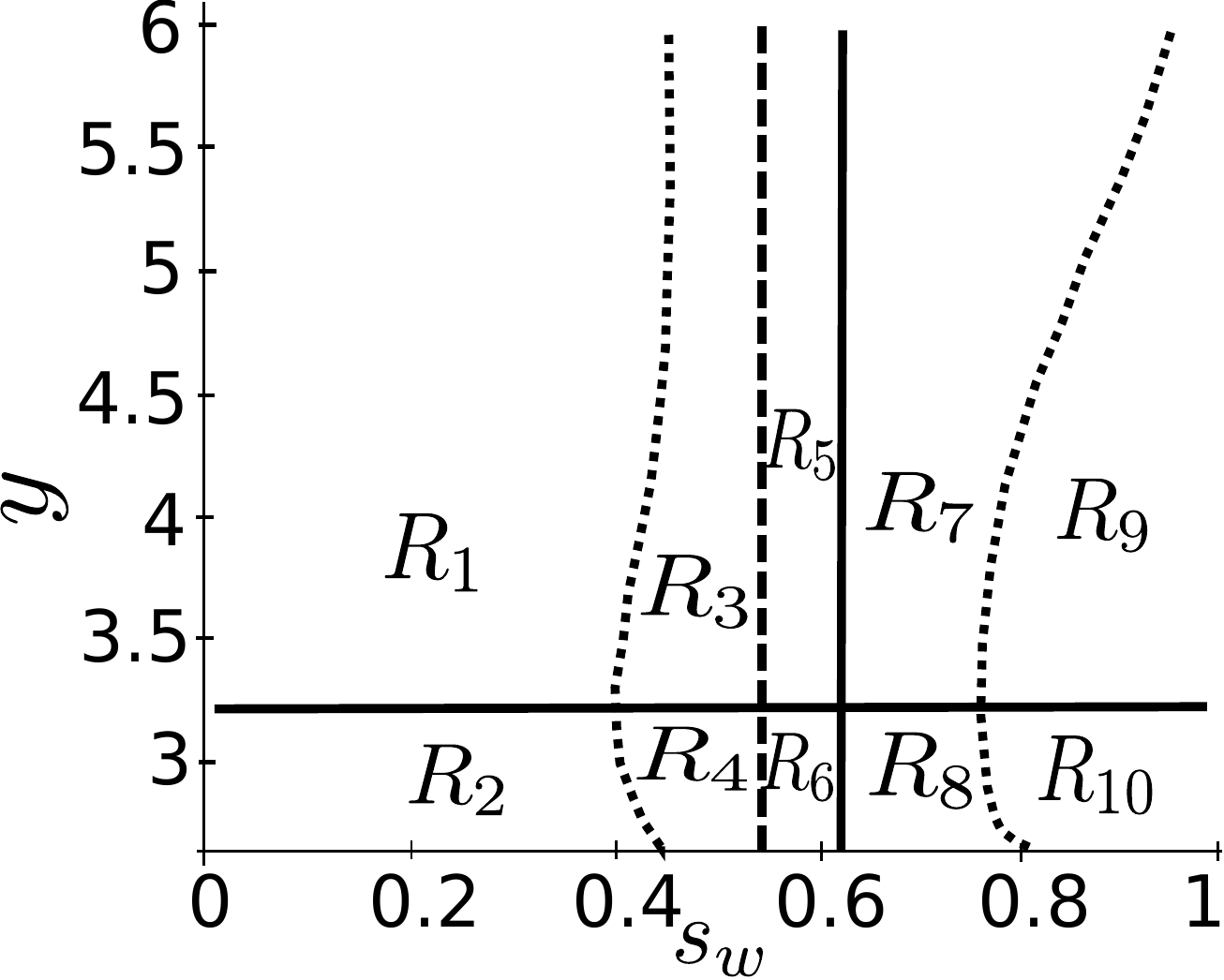}
 		\includegraphics[scale=0.42]{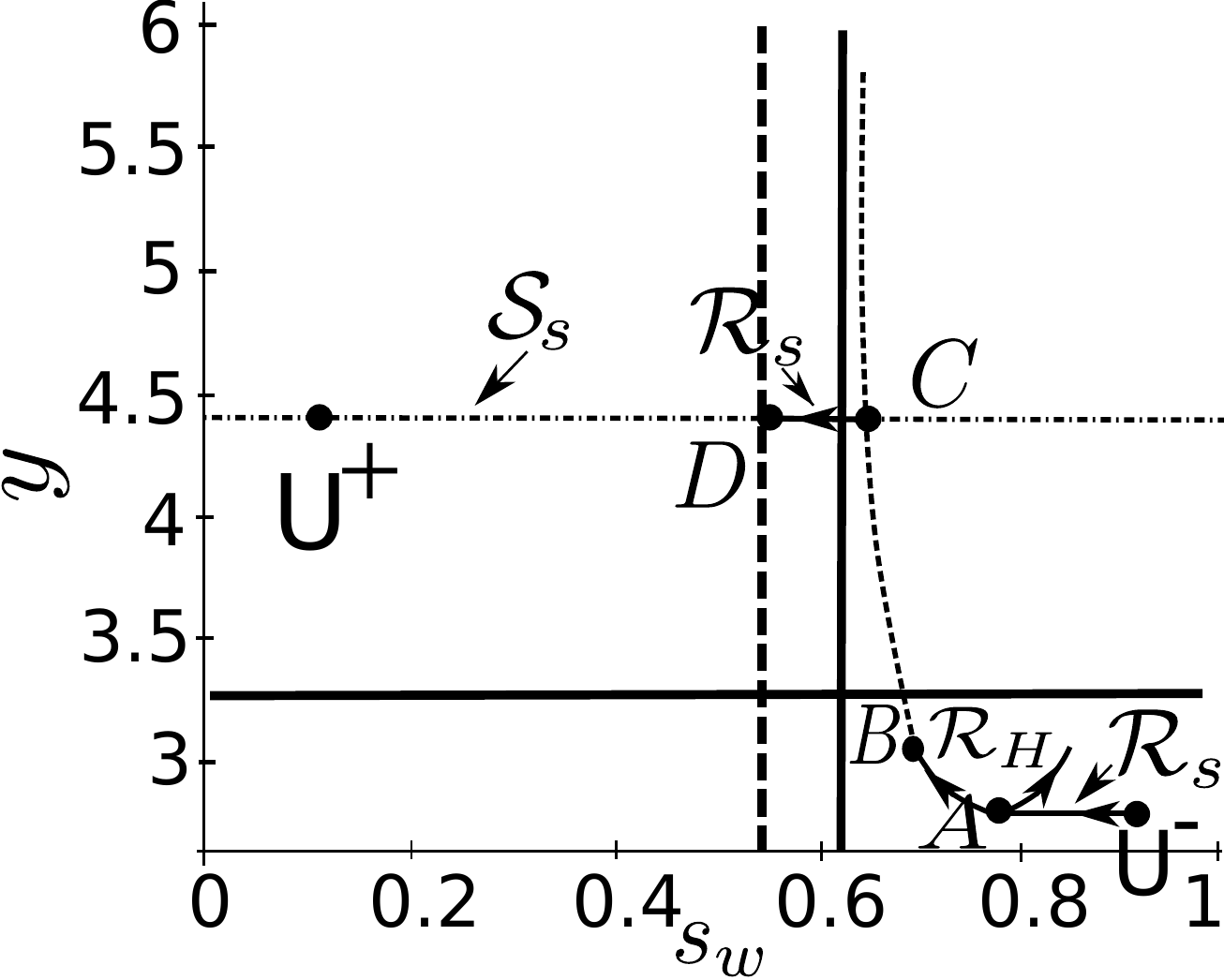}
 	\end{center}
 	\caption{
 		\textit{ a)- left}.\grdel{The ten regions $R_1$ to $R_{10}$ in the phase space.\textit{ b)- right} Wave sequence solution of the Riemann problem.  The dashed curves represent shocks and the black curves with arrows represent rarefactions. In this picture the coincidence and inflection curves are the same show in Figure \ref{phasespacea1}.
 		}
 	}\label{phasespacea1A}
 \end{figure}
 
 A Riemann solver for the system \eqref{leicons2a1} was developed to represent the rarefaction and shock curves as show in Figure \ref{fractional}. This program implements the  theoretical structures in Riemann solutions predicted by an advanced theory of the wave curve methods. With this novelty, we implement the procedure explained in this paper to traverse the coincidence locus with the rarefaction and shock curves.
 
 The self-similar elementary curves suffer modification at a state on the bifurcation loci shown in Figure $\ref{phasespacea1}$ (see explanations of the wave curve method in \cite{lambert2006riemann}). To construct the Riemann solution, we first obtain the 
 wave sequence in $(s_w,y)$ using combinations of rarefactions $\mathcal{R}_{s}$, $\mathcal{R}_H$ as well as shock curves $\mathcal{S}_{s}$, $\mathcal{S}_H$, taking into account the coincidence locus $\mathcal{C}$ together with the inflections loci $I_{s}$ and  $I_H$. The bifurcation curves split the projected space $(s_w,y)$ in ten regions, see Figure \ref{phasespacea1A} left. In particular, we see that  $\lambda_H < \lambda_s$ in subregions  $R_3$ to $R_8$   and  $\lambda_H > \lambda_s$ for states in $R_1$, $R_2$, $R_9$ and $R_{10}$.

As an example, we show the Riemann solution for the case that the left state $U^-$ is in the region $R_{10}$ and the right sate $U^+$ is in the region $R_1$. The solution is a wave sequence in the phase space  given by $ U^-\xrightarrow{\mathcal{R}_s}A\xrightarrow{\mathcal{R}_H}B\xrightarrow{\mathcal{S}_H}C
\xrightarrow{\mathcal{R}_s}D\xrightarrow{\mathcal{S}_s}U^+$; see Figure \ref{phasespacea1A} right. The state $A$ is on the coincidence curve $\mathcal{C}$, i.e. $\lambda_s(A)=\lambda_H(A)$. Moreover, $\sigma(B,C)=\lambda_H(A)$ and $\sigma(D,U^+)=\lambda_s(D)$. In the wave sequence, we drop the initial state for rarefaction curves and the initial and final states in shock curves. For the construction of the wave sequence, we use a method similar to those described in \cite{liu1974riemann,liu1975riemann}.

At the point $A$, there exists a unique eigenvector $\vec{r}_{s}=(1,0,0)$. To continue the rarefaction curve $\mathcal{R}_s$, we applied the procedure explained in Section \ref{proceM}, using the regularized manifold to build an admissible and robust direction beyond the coincidence point.  We first obtained the generalized
 Jordan chain $r_o$ and $r_1$ in \eqref{gen3} and latter used versal deformation vectors $R_o$ and $R_1$ in \eqref{verssal} to obtain the admissible direction at the coincidence characteristic speed point $A$.

 \section{Violation of genuine nonlinearity }
\label{comp}

In this section, we present an algorithm for the construction of a composite wave curve, required when a rarefaction curve stops at an inflection point, i.e. where $\triangledown \lambda \cdot r =0$. Our construction is more general, since we take into account that this inflection may happen in a coincidence point.

A discontinuity in the solution of the system of conservation laws in \eqref{eqt}
satisfies the Rankine-Huguniot locus, i.e.
\begin{equation}
F(U^+)-F(U^-)=s(G(U^+)-G(U^-)),
\label{RH}
\end{equation}
where $s$ denotes the shock speed. We assume that the generalized eigenproblem in
\eqref{ge1} has $n$ eigenvalues such that
$\lambda_1 \le \dots  \le \lambda_n.$

The resonance happens when there are two eigenvalues $\lambda_i$ and $\lambda_j$ with $i\neq j$ such that $\lambda_i(U_o)=\lambda_j(U_o)$,  where $j=i+1$ or $j=i-1$. Moreover, there exist the possibility that on the point $U_o$ the i-th family has an inflection, i.e. $\triangledown \lambda_i \cdot r_i(U_o) =0$. In this case using the Bethe-Wendrof Theorem (see \cite{wendroff1972riemann2}), it is possible to construct a combination of
a rarefaction $\mathcal{R}_i$ with a characteristic shock $\mathcal{S}_i$, satisfying the compatibility condition $\lambda_i(U^-)=s(U^-,U+)$, where $s(U^-,U+)$ denotes
the shock speed  from the state $U^-$ to $U^+$ to cross the inflection locus $\Gamma$.

However, due to the coincidence of characteristic velocities at point $U_o$, there are singularities that don't  allow the construction of a composite curve in the classical sense, as explained in \cite{liu1975riemann}. In this section, we study this phenomenon in anomalous cases appearing for the non-strictly hyperbolic setting. We provide an analysis justifying the construction of composite curve and a numerical recommendation for its implementation in a Riemann solver. 

We are interested in finding the composite curve $(U^-(\xi),U^+(\xi))$. By definition, this curve satisfies the equation \eqref{RH}, $s(U^-(\xi),U^+(\xi))=\lambda(U^-(\xi))$ and $U^-(\xi)$ obeys the ODE 

\begin{equation}
dU^-/d \xi=r(U^-(\xi)), \quad \text{with} \quad U^-(0)=U_l,
\label{edo1}
\end{equation}
where $r$ is the eigenvector associated with $\lambda$, i.e.
$DF(U^+)r=\lambda DG(U^+)r$, where $DF=\partial F/ \partial U$ and $DG=\partial G/ \partial U$. We assume that the vector $r$ has norm one and adequate sign.

Let us denote by 
\begin{equation}
E(U^-,U^+)=F(U^+)-F(U^-)-\lambda(U^-)(G(U^+)-G(U^-)).
\label{RH1}
\end{equation}
If the Jacobian $D_{U^+}E=DF(U^+)-\lambda(U^-)DG(U^+)$ is invertible in the neighborhood of a point
$(U_o^+,U^-_o)$, then by the Implicit Function Theorem there are $\rho$ and $\epsilon$ 
such that in the set $\{(U_o^+,U^-_o):||U^+-U_o^+|| < \rho ~~\text{and} ~~||U^--U_o^-|| < \epsilon\} $ the equation 
\begin{equation}
E(U^-,U^+)=0,
\label{RH1as1}
\end{equation}
has a unique solution $\widehat{U}^+=\phi(U^-)$ (or $\widehat{U}^+=U^+(U^-)$) such that $E(U^-,\phi(U^-))=0$ and
\begin{equation}
D\phi(U^-)=-(DE_{U^+}(U^-,\widehat{U}^+))^{-1} DE_{U^-}(U^-,\widehat{U}^+),
\end{equation}
or
\begin{align}
\nonumber
D\phi(U^-)=&(DF(\widehat{U}^+)-\lambda(U^-)DG(\widehat{U}^+))^{-1}  \\&\times(DF(U^-)-\lambda(U^-)DG(U^-)+(G(\widehat{U}^+)-G(U^-))\triangledown \lambda(U^-)).
\label{main1a}
\end{align}
%with $D\lambda(U^-)=(\triangledown \lambda,\triangledown \lambda)^T$.

Thus locally, for each rarefaction curve $\xi \rightarrow U^-(\xi)$ (solution of the initial value problem \eqref{edo1}), there exists a unique curve 
\begin{equation}
\xi \rightarrow U^+(\xi)=\phi(U^-(\xi)),
\label{main1}
\end{equation}
satisfying equations \eqref{RH1}. We call this curve  a \textit{composite curve}.
When the point $(U_o^+,U^-_o)$ belongs to the inflection locus, we call it \textit{the local composite curve}\gui{, otherwise we refer to it as} \textit{the non-local composite curve}.

The above definition for a composite curve requires that the determinant of $D_{U^+}E$ be different from zero. However, we present here \gui{other} possibilities to construct \gui{a} composite curve satisfying \eqref{RH1as1} with $U^-$ given as solution of \eqref{edo1}.

Let us take $K^+=K(U^+(\xi))$ and $K^-=K(U^-(\xi))$, where $K$ represents an arbitrary function. To find the composite curve, we have at least two methods: the first one consist of determining the solution of the equation \eqref{RH1} for a given $U^-(\xi)$ and choosing $\xi \in (0,\bar{\xi})$; the second is the continuation method, i.e. finding and solving an ODE associated to \eqref{RH1as1} starting at $U^-$.

\subsection{Derivation of ODE for the composite wave}

Modeling for rarefaction and shock waves can be combined to form an unique ODE that takes into account Rankine-Hugoniot and the compatibility conditions. We assume a parametrization along the wave curve $(U^-(\xi),U^{+}(\xi),s(\xi))$. Differentiating \eqref{RH} along this curve we obtain
\begin{align}
\nonumber
-(G(U^+)-G(U^-))\frac{ds}{d \xi}+\left(DF(U^+)-sDG(U^+)\right)\frac{d U^+}{d\xi}-\\
\left(DF(U^-)-sDG(U^-)\right)\frac{d U^-}{d\xi}=0.
\label{RH2}
\end{align}
We assume that the condition
\begin{equation}
\lambda_i(U^-(\xi))=s(U^-(\xi),U^+(\xi)),
\label{rh234a}
\end{equation}
holds  starting from some $\xi=\xi_l$ along the curve satisfying \eqref{RH2}. Substituting equality \eqref{rh234a} in \eqref{RH2} we have 
\begin{equation}
\left(DF(U^-)-sDG(U^-)\right)\frac{d U^-}{d\xi}=0,
\end{equation}
and
\begin{equation}
ds/d \xi= d \lambda_i/ d \xi= \triangledown \lambda_i(U^-) \cdot \tilde{r}_i(U^-),
\label{rh234b}
\end{equation}
where $\tilde{r}_i(U^-)$ is the normalized eigenvector associated with $\lambda_i(U^-)$, i.e.
$\tilde{r}_i=r_i/(l_i \cdot r_i)$, where $l_i$ is the left eigenvector associated to $\lambda_i^-$.

Using \eqref{rh234a} and \eqref{rh234b} in \eqref{RH2} we obtain ODE
\begin{equation}
\left(DF(U^+)-\lambda_i(U^-)DG(U^+)\right)\frac{d U^+}{d \xi}=  (\triangledown \lambda_i(U^-) \cdot \tilde{r}_i(U^-)) (G(U^+)-G(U^-)),
\label{RH3}
\end{equation}
with $U^+(0)=U^-$ and
\begin{equation}
\frac{d U^-}{d \xi}=\tilde{r}_i(U^-),
\label{RH4}
\end{equation}
where $U^-(\xi_l)=U_l$.
Integrating \eqref{RH4} we have
\begin{equation}
U^-(\xi)=\int_{\xi_l}^{\xi}r_i(U^-(\eta))d\eta+U_l.
\label{umas}
\end{equation}
Let us define 
\begin{equation}
A(U^+,U^-)=\left(DF(U^+)- \lambda_i^-DG(U^+)\right),~  V=U^+-U^-,~ \lambda_i^-=\lambda_i(U^-),
\label{eqr1}
\end{equation}
and study separately the cases where determinant of $A(U^+,U^-)$ is zero or not.

In the following, we study conditions under which the composite field on ODEs \eqref{RH3} and \eqref{RH4} is well defined and its singularities can be removed, allowing the construction of composite wave curves by means of our continuation method. This removal is based on the fact that, if $r_i$ is an eigenvector, so is $c r_i$ for non zero scalar values $c$; then, we only need to find appropriate scalar functions that remove the singularities of \eqref{RH3} upon multiplication. 
However, it is not always convenient to find the composite wave curve through ODEs \eqref{RH3} and \eqref{RH4}. Instead, we propose an alternative formulation: using a parametrization for the composite curve, we develop a numerical algorithm allowing its calculations.

%\gr{We emphasized in the equation \eqref{RH3} since the singularities in \eqref{RH4} can be removed from the fact that if ${r}_i$ is eigenvector then $c{r}_i$ as well, where $c$ is certain known expression.}

\subsubsection{The generic violation case}

To calculate the eigenvalues and eigenvector of the generalized eigenvalues problem
in \eqref{ge1}, we use the fact for the Generalized Schur decomposition of the Jacobian $DF$ and $DG$ of flux and accumulation functions there exist unitary matrix $Q$ and $Z$ such that the matrix $Q(DF)Z$ and $Q(DG)Z$ are upper triangular $T$ and $S$ with element in the diagonal $t_{ii}$ and $s_{ii}$. Then the eigenvalues are given by $\lambda_i=t_{ii}/s_{ii}$ with $s_{ii}, \neq 0$, see \cite{datta2010numerical}.  

Here, we study the qualitative behavior the solution of ODE system \eqref{RH3}-\eqref{RH4} when the determinant $det(A(U^+,U^-)) \neq 0$ for all point along the curve satisfying \eqref{RH2}-\eqref{rh234b}, which is called the composite wave curve. 

Since determinant $det(A(U^+,U^-))$ can be written as 
\begin{equation}
det(A(U^+,U^-))=det(S)(\lambda^+_1-\lambda_i^-)  \cdots  (\lambda^+_n-\lambda_i^-),
\label{det1}
\end{equation}
the condition $det(A(U^+,U^-)) \neq 0$ implies that there is no $\xi$ along the composite wave curve  such that $\lambda_j(U^+(\xi))=\lambda_i(U^-(\xi))$ for $i=1,\cdots,n$ and $j=i-1$ or $j=i+1$.

From \eqref{RH3} we obtain
\begin{equation}
\frac{d U^+}{d \xi}=  \frac{\triangledown \lambda_i^- \cdot \tilde{r}_i(U^-)}{det(A(U^+,U^-))}Adj(A(U^+,U^-))(G(U^+)-G(U^-)).
\label{RH33}
\end{equation}
where $Adj$ denotes the adjugate matrix, see \cite{gantmacher1960theory}.

From expressions \eqref{det1} for the determinant and \eqref{RH33} for the composite curve, we see that resonance phenomena can lead to composite fields that are not well-defined. We address this situation by analyzing the equivalent problem of a collision between shock and rarefaction wave curves.

Taking the difference between \eqref{RH33} and \eqref{RH4}, and substituting $U^+=V+U^-$, we obtain
\begin{equation}
\frac{d V}{d \xi}=  \frac{\triangledown \lambda_i^- \cdot \tilde{r}_i(U^-(\xi))}{det(A(V+U^-,U^-))}Adj(A(V+U^-,U^-))V-\tilde{r}_i(U^-(\xi)),
\label{RH33a}
\end{equation}
where $U^-(\xi)$ is given by \eqref{umas}. Thus we have an ODE for $V$ starting at $\xi=\xi^-$, where $V=0$.
Equation \eqref{RH33a} can be rewritten as
\begin{align}
\nonumber
\frac{d V}{d \xi}=&  \frac{\triangledown \lambda_i^- \cdot \tilde{r}_i(U^-(\xi))}{det(A(V+U^-,U^-))}Adj(A(V+U^-,U^-))  \\
 &\times\left(V-\frac{1}{\triangledown \lambda_i^- \cdot \tilde{r}_i(U^-(\xi))}A(V+U^-,U^-)\tilde{r}_i(U^-(\xi))\right),
\label{RH33a1}
\end{align}

Then, we use ODE \eqref{RH33a1} to do the analysis of the solution of the ODE system \eqref{RH3}-\eqref{RH4}. By inspection of right side of \eqref{RH33a1}, it is possible to verify that singular points are the points $(U^-(\xi),U^+(\xi))$ where $r_i(U^-(\xi)) \in Ker(A(U^+(\xi),U^-(\xi)))$ and $\triangledown \lambda^- \cdot r_i(U^-(\xi))=0$.

Also, it is possible to verify that if $\triangledown \lambda^- \cdot r_i(U^-(\xi))=0$ then
$det(A(U^+,U^-))=0$ \gui{and vice versa}. Therefore, there are no singular points under the assumption that $det(A(U^+,U^-)) \neq 0$. \gui{Under this assumption} it is possible to construct the composite curve using  \eqref{RH3} and \eqref{RH4}\gui{, but it is necessary define} the field at $\xi=0$ or $U^+=U^-$ appropriately.	Notice that the curve $U^-(\eta)$ and $U^+(\eta)=U^-(\eta)$ is solution of \eqref{RH1as1}, which we call a \textit{trivial composite curve}. Thus, we can define $dU^+/d\xi=-r_i(U^-)$ at $U^+(0)=U^-$ when $ \triangledown (\triangledown \lambda^- \cdot r_i) \cdot r_i(U^-(\xi)) \neq 0$ (\gui{see \cite{ancona2001note}} for the proof).
	In this way the solution of ODE \eqref{RH33a1} do not reproduce the trivial case.

\subsubsection{The singular violation case}
\label{ape1}

In dedicate this subsection to the study of some singular cases appearing in the construction of composite wave curves. We analyze under what conditions the composite field in \eqref{RH33} is well defined and how appearing singularities may be removed. The idea behind this study is to simultaneously prove the existence of composite curves and to indicate how they are constructed through a continuation method with an appropriate parametrization.

We study the case of \gui{isolated} singular points where $det(A(U^+,U^-))=0$ \gui{,or equivalently when} the Jacobian $D_{U^+}E$ is not singular. In this case \gui{it} is not possible to obtain the diffeomorfism between variables $U^+$ in function of $U^-$ by applying the Implicit Function Theorem. This \gui{also} means that \gui{it} is not possible to use \gui{equations} \eqref{RH3} and \eqref{RH4} to obtain the \gui{other} branch of \gui{the} composite curve. However, in this case we find a local parametrization \gui{for the} composite curve close \gui{to these} singular points in another system of coordinates.

First, we assume that there is only one value for
index $j$ such that $\lambda_{j}((U^+)^*)-\lambda_i((U^-)^*)=0$ at distinct points $(U^+)^*$ and $(U^-)^*$. Let $l_{j}^+$ be the left eigenvector associated to an eigenvalue $\lambda_{j}$ of the generalized eigenproblem in \eqref{ge1}.

Multiplying \eqref{RH3} by the left eigenvector $l_{j}^+=l_{j}(U^+)$ of the generalized eigenproblem \eqref{ge1}, we obtain

\begin{equation}
l_{j}^+ \cdot \left(DF(U^+)-\lambda_i(U^-)DG(U^+)\right)\frac{d U^+}{d \xi}=  (\triangledown \lambda_i(U^-) \cdot \tilde{r}_i (U^-)) ( l_{j}^+ \cdot (G(U^+)-G(U^-)).
\label{RH3a1}
\end{equation}
Using that $l_{j}^+ \cdot DF(U^+)=\lambda_{j}^+ l_{j}^+DG(U^+)$, we obtain from \eqref{RH3a1}

\begin{equation}
\left(\lambda_{j}^+-\lambda_i^-\right) \left(l_{j}^+ DG(U^+)\cdot \frac{d U^+}{d \xi}\right)=  (\triangledown \lambda_i(U^-) \cdot \tilde{r}_i (U^-))( l_{j}^+ \cdot (G(U^+)-G(U^-))).
\label{RH3a2}
\end{equation}

In this situation, we have to analyze the following possibilities:
\begin{itemize}
	
	\item[a)] $\triangledown \lambda_i^- \cdot r_i((U^-)^*) \neq 0$,

	\item[b)]  $\left[\triangledown \lambda_i((U^-)^*) \cdot \tilde{r}_i ((U^-)^*)\right] \left[l_{j}^+ \cdot (G((U^+)^*)-G((U^-)^*))\right] \neq 0$,
	
	\item[c)]  $\left[\triangledown \lambda_i((U^-)^*) \cdot \tilde{r}_i ((U^-)^*)\right] \left[l_{j}^+ \cdot (G((U^+)^*)-G((U^-)^*))\right] = 0$,
	
	\item[d)] $\triangledown \lambda_i \cdot \tilde{r}_i   \equiv 0$  on a submanifold $\Sigma$ of codimension $1$,
	
	\item[e)]there exists a points $(U^-)^*$ such that $\lambda_{i+1}((U^-)^*)-\lambda_i((U^-)^*)=0$ with $\triangledown \lambda_i \cdot r_i((U^-)^*)=0$.

\end{itemize}

\textbf{Case a)} \gui{We consider the case} when  $\triangledown \lambda_i^- \cdot r_i((U^-)^*)$  is not zero and there exist \gui{distinct points $(U^+)^*$, $(U^-)^*$ and only one value for $j$} such that $\lambda_{j}((U^+)^*)-\lambda_i((U^-)^*)=0$. Using \eqref{RH3a2}, we conclude that if $\lambda_{j}((U^+)^*)-\lambda_i((U^-)^*)=0$ at some point $U^*=((U^-)^*,(U^+)^*)$ and  $\triangledown \lambda_i \cdot r_i((U^-)^*) \neq 0$, then  $l_{j}^+(U^*) \cdot (G((U^+)^*)-G((U^-)^*))=0$, i.e. the vector $G((U^+)^*)-G((U^-)^*)$ is orthogonal to the left eigenvector $l_{j}((U^+)^*)$. This is a necessary condition for a well defined composite field at the points $(U^+)^*$, $(U^-)^*$, since in the neighborhood of these points the eigenvector $\tilde{r}_i$ can be scaled by $\bar{r}_i=(\lambda_{j}^+-\lambda_i^-)\tilde{r}_i$ so that it is well behave at point $U^*$.

On the other hand, if  $l_{j}^+ \cdot (G((U^+)^*)-G((U^-)^*))=0$ and $l_{j}^+ DG(U^+)\cdot \frac{d U^+}{d \xi} \neq 0$ then $\lambda_{j}((U^+)^*)-\lambda_i((U^-)^*)=0$ therefore $det(A((U^+)^*,(U^-)^*)=0$. \\

%Also it possible to verify that if $d U^+/d \xi \in N(A(U^-,U^+))$ then there exist $j$ such that $\lambda_j(U^+)=\lambda_i(U^-)$ therefore $det(A)=0$. If $d U^+/d \xi \in N(A(U^-,U^+))^\perp$ then $det(A)\neq0$.

\textbf{Case b)} There exists a bifurcation point $((U^+)^*,(U^-)^*)$ when $(\triangledown \lambda_i((U^-)^*) \cdot \tilde{r}_i ((U^-)^*))\times( l_j^+ \cdot (G((U^+)^*)-G((U^-)^*)) \neq 0$ and $\lambda_{j}((U^+)^*)-\lambda_i((U^-)^*)=0$. In this case, it is not always possible to solve system \eqref{RH3}-\eqref{RH4} starting at this point and we need to regularize the characteristic field to construct the composite curve.

In this situation, it is possible to obtain a parametrization for the composite wave curve where the parameter is one of the components of vector $U^+$. We detail  this result in the following Lemma. We denote $U^+=(U_1^+,\ldots,U_j^+,\ldots,U_n^+)$ by $U^+=(W,U_j^+)$ where
$W=(U_1^+,\ldots,U_{j-1}^+,U_{j+1}^+,\ldots,U_n^+)$.
	
	\begin{lemma}
		\gui{For $n \ge 2$,} let $(U^+)^*$ be a point on the Hugoniot locus constructed from $(U^-)^* \neq (U^+)^*$. Assume that there exists only one index value $j$ such that
		\begin{itemize}
			\item[i)] $\lambda_i((U^-)^*)=s((U^-)^*,(U^+)^*)=\lambda_j((U^+)^*)$, 
			\item[ii)] $\triangledown \lambda_i \cdot \tilde{r}_i ((U^-)^*) \neq 0$,
			\item[iii)] $l_{j}^+ \cdot (G((U^+)^*)-G((U^-)^*)) \neq 0   $.
		\end{itemize}
        	Let $U^-(\alpha)$ be the solution of \eqref{RH4} with $U^-(0)=(U^-)^*$. Here $U_j^+$ is the j-$th$ component of the vector $U^+ \in \Re^n$. Then, there exists a parametrization 
		$U_j^+ \rightarrow \alpha(U_j^+)$ and $U_j^+ \rightarrow U^-(\alpha(U_j^+))$ satisfying
		\begin{equation}
		\dfrac{dU^-}{d U_j^+} \dfrac{dU_j^+}{d \alpha}= r_i(U^-(\alpha(U_j^+))) \text{, with } U^-(0)=(U^-)^*.
		\end{equation}
		 Let us denote by $W$ the variable consisting of the vector $U^+$ without its j-$th$ component. Then additionally, there is a parametrization $U_j^+ \rightarrow W(U_j^+)$ in a neighborhood of
		the point $((U^-)^*,(U^+)^*)$ satisfying 
			\begin{equation}
			E(U^-(\alpha(U_j^+)),W(U_j^+),U_j^+)=0, 
			\label{RH12a1}
			\end{equation}
			where $E(U^-,U^+)$ is defined by \eqref{RH1} and satisfies \eqref{rh234a}, i.e.
			\begin{equation}
			\lambda_i(U^-(\alpha(U_j^+)))=s(U^-(\alpha(U_j^+),U^+(U_j^+)).
			\label{rh234a121}
			\end{equation} 
		\label{lemv1}	
	\end{lemma}
	
	\begin{proof}
	%	These conclusions follow from straightforward application of equation \eqref{RH3a2}.
	
	From hypothesis (i), the matrix $D_{U^+}E=DF((U^+)^*)-\lambda((U^-)^*)DG((U^+)^*)$ has rank $n-1$ and is therefore equivalent to the matrix $ diag(s_1 (\lambda_1^+-\lambda_i^-),\cdots,0,\cdots,\\s_n (\lambda_n^+-\lambda_i^-))$, with zero in the j-$th$ position. We consider $E(U^-,U^+)$ as defined in \eqref{RH1}.
	
	We define the map $ S : (\alpha, W,U_j^+) \rightarrow E(U^-(\alpha),W,U_j^+,)$ with total differential at point $(0,(U^+)^*)$ given by
	\begin{equation}
	dS =  D_{W}E dW +D_{U_j^+}E dU_j + D_{\alpha}E d\alpha=0.
	\end{equation}
	with
	\begin{equation}
D_{\alpha}E =\triangledown \lambda_i \cdot \tilde{r}_i ((U^-)^*) (G((U^+)^*)-G((U^-)^*)),
	\end{equation}
\begin{equation}
D_{W}E=D_{W}F((U^+)^*)-\lambda((U^-)^*)D_{W}G((U^+)^*),
\end{equation}	
where $D_{W}E$ does not \gui{contain} the $j$-th column of the Jacobian $D_{U^+}E$,
and
\begin{equation}
D_{U_j^+}E=D_{U_j^+}F_j((U^+)^*)-\lambda((U^-)^*)D_{U_j^+}G_j((U^+)^*).
\label{key1}
\end{equation}	
By hypothesis (ii), (iii) and since $rank(D_{W}E)=n-1$, we have $rank(D_{W,\alpha}S)=n$ with 
$D_{\alpha,W}S = [D_{\alpha}E \ | \ D_{W}E]$. By the Implicit Function Theorem, there exists an open set $I_1$ of $\Re$ containing $(U_j^+)^*$ and an unique continuously differentiable function $h: I_1 \rightarrow \Re^{n}$, $h=(h_1,h_2)$ such that
\begin{equation}
S(\alpha(U_j^+), W(U_j^+),U_j^+)=E(U^-(h_1(U_j^+)),h_2(U_j^+),U_j^+)=0.
\end{equation}
Thus we obtain a parametrization $U_j^+: I_1 \rightarrow h_1(U_j^+)=\alpha(U_j^+)$
and $U^+: I_1 \rightarrow W(U_j^+)=h_2(U_j^+)$.

In order to obtain the parametrization for $\alpha$ in function of $U_j^+$, we solve the ODE
\begin{equation}
\left[\frac{d \alpha}{ d U_j^+},\frac{d h_2(U_j^+)}{ d U_j^+}\right]^T =\sum_{i=1}^{n} ([D_{\alpha,W}S]^{-1})_{ki}D_{U_j^+}E(U^-(\alpha(U_j^+)),h_2(U_j^+)),
\end{equation}
with $k=1,\ldots,n$,  $\alpha(0)=0$ and $D_{U_j^+}E$ given by \eqref{key1} ($d/dU_j^+$ denotes the total differentiation respect to the one dimensional variable $U_j^+$).

	\end{proof}
	
	Under the hypothesis of Lemma \eqref{lemv1}, Furtado studied in \cite{furtado1991structural} the situation of $\triangledown \lambda_i(U^-) \cdot \tilde{r}_i (U^-)  \neq 0$, $( l_j^+ \cdot (U^+-U^-)) \neq 0$ and $rank(D_{U^+}E)=n-1$. This work provides a method to construct a parametrization for Hugoniot and rarefaction wave curves which is also useful in the construction of composite curves.\\

\textbf{Case c)} There exists a bifurcation point $((U^+)^*,(U^-)^*)$ when $(\triangledown \lambda_i((U^-)^*) \cdot \tilde{r}_i ((U^-))^*)( l_j^+ \cdot (G((U^+)^*)-G((U^-)^*))= 0$ and $\lambda_{j}((U^+)^*)-\lambda_i((U^-)^*)=0$.

In this case, we cannot guarantee the existence of composite curves since the Implicit Function Theorem is not applicable to equation \eqref{RH1as1} with
$E$ given by \eqref{RH1}. That is because $rank(D_{U^+}E) < n$, since $det(A(U^+,U^-))=0$. Consequently, it is not possible to equation \eqref{RH1as1} by means of a local diffeormophism between states $U^-$ and $U^+$.

A solution to this difficulty was given by \cite{muller2001existence} in the particular case of $\triangledown \lambda_i(U^-) \cdot \tilde{r}_i (U^-) =0$ and $rank(D_{U^+}E)=n-1$. In that work, Theorem $1$ states that it is possible to construct the composite wave curve for the strictly hyperbolic case in a simple degeneration point where $\triangledown \lambda_i^- \cdot r_i(U^-)=0$ and $\triangledown_{U} (\triangledown \lambda_i^- \cdot r_i) \cdot r_i(U^-) = 0$. The proof given in \cite{muller2001existence} also serves to show the existence of a local parametrization for the composite curve in the case analyzed here, but does not provide a robust numerical algorithm to obtain it.\\

 The method described in \cite{furtado1991structural}, together with the Lyapunov-Schmidt reduction principle (see \cite{golubitsky2012singularities}), is useful for the construction of composite wave curves for the case of $l_{j}^+ \cdot (G((U^+)^*)-G((U^-)^*)) = 0   $, where  $l_{j}^+=l_{j}((U^+)^*)$. This construction is summarized in the Lemma below.
 
 \begin{lemma}
 	Let $(U^+)^*$ be a point on the Hugoniot locus based on $(U^-)^*\neq(U^+)^*$. Assume that
 	\begin{itemize}
 		\item[i)] there exists only one index value $j$ such that $\lambda_i((U^-)^*)=s((U^-)^*,(U^+)^*)=\lambda_j((U^+)^*)$,
 		\item[ii)] $\triangledown \lambda_i \cdot \tilde{r}_i ((U^-)^*) \neq 0,$
 		\item[iii)] $l_{j}^+ \cdot (G((U^+)^*)-G((U^-)^*)) = 0 $,
 		\item[iv)] $l_{j}^+ \cdot ((U^+)^*) \neq 0   $.
 	\end{itemize}

 	Then, there is a parametrization for the integral curve
 	$\alpha \rightarrow U^-(\alpha)$ satisfying \eqref{RH4}  and $U^-(0)=(U^-)^*$. Moreover, there is also a parametrization 
 	$\beta \rightarrow U^+(\beta)$ with $ U^+(\beta^*)=(U^+)^*$  satisfying  \eqref{RH1}, i.e.
 	\begin{equation}
 	E(U^-(\alpha),U^+(\beta))=0,
 	\label{RH12a}
 	\end{equation}
 	where $E(U^-,U^+)$ is defined in \eqref{RH1} and satisfies \eqref{rh234a}, i.e.
 	\begin{equation}
 	\lambda_i(U^-(\alpha))=s(U^-(\alpha),U^+(\beta)).
 	\label{rh234a12}
 	\end{equation}
  \end{lemma}
 
 \begin{proof}
 	Let $U^-(\alpha)$ be the solution of \eqref{RH4} with $U^-(0)=(U^-)^*$ 	
 	and such that $\triangledown \lambda_i \cdot \tilde{r}_i (U^-(\alpha)) < 0$. We consider $E(U^+,U^-)$ as defined in \eqref{RH1} and the map $ S : (\alpha, U^+) \rightarrow E(U^+,U^-(\alpha))$ with total differential at point $(0,(U^+)^*)$ given by
 	\begin{equation}
dS =  D_{U^+}E dU^+ + \triangledown \lambda_i \cdot \tilde{r}_i ((U^-)^*) (G((U^+)^*)-G((U^-)^*)) d\alpha=0.
 	\end{equation}
 From hypothesis $i)$, we see that $J=D_{U^+}E=DF((U^+)^*)-\lambda((U^-)^*)DG((U^+)^*)$ has rank $n-1$ when restricted to subspace
 	\begin{equation}
 	V=\{U \in \Re^n : l_j((U^+)^*) \cdot U=0\}.
 	\end{equation}
 Additionally, from iii) we verify that it is not possible to obtain a parametrization for variable $U^+$ as a function of $\alpha$ from the Implicit Function Theorem. We construct such a parametrization by means of an appropriate partition of the space and an additional equation.
 		
 We take a partition
 \begin{equation}
\Re^n=V\oplus V^\perp
 \end{equation}	
with $dimV=dimJ=n-1$ and $dimV^\perp=1$. 	
 	
Since $(U^+)^* \notin V$ by hypothesis $iv)$, for each $U \in \Re^n $ there are two unique $U^\top \in V$ and $U^\bot \in V^\perp$ such that $U=U^\top+U^\bot$, where $U^\bot$ can be parametrized as
\begin{equation}
U^\bot(\beta)=(U^+)^*+(\beta-\beta^*)r_i((U^+)^*),
\label{com1}
\end{equation} 
and $r_i$ is the right eigenvector associated to eigenvalue $\lambda_i((U^+)^*)$ of the generalized eigenproblem \eqref{ge1}. 

The conclusion that there is a unique $\beta=\beta^*-((U^+)^* \cdot U^\top)/(r_i((U^+)^*) \cdot U^\top)$, where $r_i((U^+)^*) \cdot U^\top \neq 0$, follows from the facts that: ($i$) for each $U^\bot \in V^\perp$, the product $U^\bot \cdot U^\top =0$; ($ii$) the corresponding state $(U^+)^*$ satisfies $(U^+)^* \notin V$ and $l_i \cdot r_i ((U^+)^*) \neq 0$ (where $l_i$ is the left eigenvector associated to $\lambda_i((U^+)^*)$).	
 	
Let $P_V$ be the projection of a subset of points satisfying $E(U^+,U^-)=0$ onto subspace $V$. We consider the system of equations for variables $(\beta,\alpha,U^\top)$
\begin{equation}
P_V(E(U^-(\alpha),(U^+)^*+(\beta-\beta^*)r_i((U^+)^*)+U^\top))=0,
\label{RHG1}
\end{equation}
and 
 \begin{equation}
M(U^\top)=l_j((U^+)^*) \cdot U^\top=0,
\label{RHG2}
 \end{equation}	
where $E$ is given by \eqref{RH1} and $U^\top \in V$. 

System of equations \eqref{RHG1}-\eqref{RHG2}
defines a map $\varGamma: \Re^{n+1} \rightarrow \Re^{n}$ expressed as $\varGamma(\alpha,\beta,U^\top)=(P_VE(U^-(\alpha),U^\bot(\beta)+U^\top),M(U^\top))$, where $U^\bot(\beta) \in V^\perp$ follows parametrization \eqref{com1} and $U^\top \in V$.  
We verify that, at the point $(\beta^*,0,P_V((U^+)^*))$ with $(U^+)^*=P_V((U^+)^*)+((U^+)^*)^\perp$, 
\begin{equation}
D_UP_VE=P_VJ,
\end{equation}
\begin{equation}
D_{\alpha}P_VE=\triangledown \lambda_i \cdot \tilde{r}_i ((U^-)^*)  P_V(G((U^+)^*)-G((U^-)^*)),
\end{equation}
\begin{equation}
D_{\beta}P_VE=[DF((U^+)^*)-\lambda((U^-)^*)DG((U^+)^*)]r_i((U^+)^*).
\end{equation}
Therefore $rank(D_UP_VE)=n-1$, and since $\triangledown \lambda_i \cdot \tilde{r}_i ((U^-)^*)  \neq 0$ we have $D_{\alpha}P_VE \neq 0$ and $D_{\beta}P_VE =0$.
Also note that
\begin{equation}
D_U(M(U^\top))=P_Vl_j((U^+)^*).
\end{equation}

Then, we have the jacobian matrix $D_{\beta,\alpha,U}\Gamma = [D_{\beta}\Gamma \ | \ D_{\alpha,U}\Gamma]$ of size $n \times n+1$ written as
$$
\begin{bmatrix}
0_{n-1,1}  & \triangledown \lambda_i \cdot \tilde{r}_i ((U^-)^*) P_V(G((U^+)^*)-G((U^-)^*)) & P_VJ \\
0 & 0 &P_V l_j^+
\end{bmatrix},
$$
where $0_{n-1,1}$ stands for the null column vector of $n-1$ elements.

Since $D_{\alpha,U}\Gamma$ has rank $n$, then by the Implicit Function Theorem there exists an open set I of $\Re$ containing $\beta^*$ and a unique continuously differentiable function $g: I \rightarrow \Re^{n+1}$ such that
$g(\beta^*)=(0,P_V((U^+)^*))$ and 
\begin{equation}
\Gamma(\alpha(\beta),g(\beta))=0.
\end{equation}
Thus we obtain a parametrization $\alpha: I \rightarrow \alpha(\beta)=g_1(\beta)$
and $U^+: I \rightarrow U^+(\beta)=g_2(\beta)$.

 \end{proof}	
	
We obtain the parametrization $U^+: I \rightarrow U^+(\beta)=g_2(\beta)$ numerically by solving system \eqref{RHG1}-\eqref{RHG2} with a Quasi-Newton method for each $\alpha$.\\

\textbf{Case d)} If $\triangledown \lambda_i(U^-) \cdot \tilde{r}_i (U^-)  \equiv 0$  on a submanifold $\Sigma$ of codimension $1$, then we have 
\begin{equation}
\left(DF(U^+)-\lambda_i(U^-)DG(U^+)\right)\frac{d U^+}{d \xi}=  0.
\label{RH3a12}
\end{equation}
If $\frac{d U^+}{d \xi}\neq 0$, then necessarily $det(A(U^+,U^-))=0$ and therefore there are $j=i+1,i-1$ such that $\lambda_j(U^+)=\lambda_i(U^-)$ and can be taken as
\begin{equation}
\frac{d U^+}{d \xi}=r_j(U^+),
\label{edo5}
\end{equation}
with $U^- \in \Sigma$ and $r_j$ is the generalized eigenvector associated to $\lambda_j$ .
%We obtain the composite wave curve by integration of \eqref{edo5} along $\Sigma$.

\textbf{Case d)}This case reduce to
\begin{equation}
\left(DF((U^-)^*)-\lambda_i((U^-)^*)DG((U^-)^*)\right)\frac{d U^+}{d \xi}=  0.
\label{RH3a12}
\end{equation}
If $\frac{d U^+}{d \xi}\neq 0$, then $det(A(U^+,U^-))=0$ and therefore we obtain the composite field at point $(U^-)^*$ as
\begin{equation}
\frac{d U^+}{d \xi}=r_j((U^-)^*).
\label{edo5}
\end{equation}

\subsection{Application to the Quadratic Corey permeability model}

Here we apply the construction of the composite wave curve shown in Section  \ref{comp} to the Riemann problem in the Corey Quad model (see \cite{azevedo1995multiple}). Numerical implementation were done with the exact Riemann solver RPN (http://rpn.fluid.impa\\.br/). The program is based on the elementary wave curves, i.e. rarefaction and shock curves. Rarefactions are the integral curves along the right eigenvector whose direction corresponds to increasing eigenvalues. Admissible shocks are obtained from the Hugoniot locus, which is obtained numerically by the continuation and quasi-newton methods.
This exact Riemann solver contains the construction method of the Riemann solution taking into account the bifurcation structures, such as the
inflection, secondary bifurcation, hysteresis, double contact and coincidence loci. ODE solver and algebraic reconstruction methods of curves are used. Moreover, other useful curves as extension through and continuation are used in the construction of wave curves (see definition of all this concepts in \cite{matos2015bifurcation} and bibliography cited there in). 

In all cases, the program can be adjusted to any particular model. However, the Riemann solver allows a major degree of generalization such that the algorithms can be extended to solve the Riemann problem for any system of conservation laws. Some examples using this program  are \cite{Matos2015,Matos2016,matos2017compositional}

This exact Riemann solver is useful as a validation tool of the numerical schemes and vise versa. In this sense, a numerical model based on finite difference schemes was incorporated which enables a cross validation. However, Riemann solver presented numerous advantages because it allows the determination of structure in the solution at different stages and their relation with physical phenomena associated with the model. The bifurcation analysis serves to determined those region where abrupt changes of solution arise. For example, for the wave curve method within this solver can be used to estimate the optimum initial condition for oil recovery.  
 
 As an application of the solver, we present a numerical example where we consider flow fields when singularities appear in the construction of local and non-local composite curves. We use the quadratic model consisting of the Cauchy problem for the system of conservation laws
\begin{equation}
\frac{\partial u}{\partial t}+\frac{\partial}{\partial x}\left[\frac{\alpha u^2}{\alpha u^2+\beta v^2+\gamma(1-u-v)}\right]=0,
\end{equation}
\begin{equation}
\frac{\partial v}{\partial t}+\frac{\partial}{\partial x}\left[\frac{\beta v^2}{\alpha u^2+\beta v^2+\gamma(1-u-v)}\right]=0,
\end{equation}
where $\alpha$, $\beta$ and $\gamma$ are positive constant and $(u,v) \in \Omega:=\{0 < u+v < 1, u,v >0 \}$

The general construction principle
for the Riemann solution is based on the scale-invariance of its
solution. In general, the Riemann solution is composed of different wave curves in the state space and waves in the time-space continuum which correspond to different characteristic velocities. Since the Corey Quad model is non-strictly hyperbolic and is non-genuinely nonlinear, the construction
of Riemann solution requires that the composite wave to crosses the inflection locus.
 We have two families, each corresponding to eigenvalue
$\lambda_i$ $(i=1,2)$, and  only one coincidence point called the umbilic point (see \cite{azevedo1995multiple}).
The Riemann solution is composed of intersection points of the
different wave curves each corresponding to a characteristic field in the phase space
connecting the two initial states. To find these points and make appropriate changes to curve is one of the challenges
of the Riemann solver. The Riemann solver is capable to determinate wave curves in the phase state. However,
for non-strictly hyperbolic case the classical waves types as described by \cite{liu1975riemann} are not sufficient to construct the wave curve. Therefore, new way as composite curves or transitional shock must be considered. 

With these tools in mind, we exemplify the Riemann solution for two particular Riemann problems. Let $L=(0.3389,0.5906)$ be the left state and take the right states $R_1=(0.0794,0.8581)$ and $R_2=(0.5819,0.2763)$.  To solve the Riemann problem defined by $L$ and $R_1$, we first choose a forward wave curve of the first family starting at state $L$ and a backward wave curve of the second family from the state $R_1$. 
We continue the forward wave curve with a sequence of admissible waves until we arrive
at the physical boundary. Clearly, in the construction of forward wave curve, we provide a procedure to construct a composite wave curve when necessary to cross the bifurcation curves. Then we determine the intersection point between these forward and backward wave curve and check if the corresponding state is admissible (see Figure \ref{figF1}).

A Riemann profile is produce in the x-t plane with each point corresponding to only one in the spaces of states. In this representation, the Riemann solution from state $L$ to $R_1$ consists of: from $L$ to $A$ a rarefaction of the family 1, from $A$ to a constant state $B$ a shock of the family 1 with $\sigma(A,B)=\lambda_1(A)$, and finally from the state $B$ to $R_1$ by a shock of the family two (see profile in Figure \ref{figFa1}).

To construct the Riemann solution from $L$ to $R_2$ we take from $L$ to $A$ a rarefaction of the family 1, from $A$ to $C$ a shock of the family one such that $\sigma(A,C)=\lambda_1(A)$ followed by a shock of the family two from the constant state $B$ to $R_2$.

The shocks from $A$ to $B$ and $A$ to $C$ is obtained from the construction of the composite curve starting at state $L$. Such curves in this case have two branches. The branch containing state $C$ is called local because it starts at the inflection locus, while the other branch that contains the state $B$ is the non-local branch.

The construction of the composite curves by the continuation method consists of the following \textbf{Algorithm} \textbf{4.1}: \\

\gui{\noindent{\textbf{Algorithm 4.1:}}}
\begin{itemize}
	\item[1)] Construct a rarefaction of the family $i$ solving ODE \eqref{RH4} starting at $U_L$ belongs to the region  $\Omega_1=\{U: \lambda_i(U) < \lambda_{i+1}(U)\}$ and stopping at the inflection point $U^+_{if}$,
	
	\item[2)] Calculate the extension point $U^+_{eif}$ (belongs to the region  $\Omega_2=\{U: \lambda_i(U) > \lambda_{i+1}(U)\}$) of the above rarefaction with a shock satisfying $\sigma(U^-_{if}-\epsilon,U^+_{eif})=\lambda(U^-_{if}-\epsilon)$, for fixed small
	$\epsilon$ such that the rarefaction point $U^-_{if}-\epsilon$ belongs to the region $\Omega_1$,
	
	\item[3)] Construct a local composite wave solving the ODE \eqref{RH3}-\eqref{RH4} starting at the initial point  $(U^+_{if}-\epsilon,U^-_{eif})$.	The initial direction of integration of rarefaction is such that $\lambda_i$ decreases, i.e. $\triangledown \lambda_i \cdot r_i < 0$ with initial direction $r^o_i$ and
	the initial direction of composite field in \eqref{RH3} is $-r^o_i$,
	
	\item[4)] Calculate the non-local initial point on the secondary bifurcation $U^+_{cnl}$ such that $\sigma(U^-_{il},U^+_{cnl})=\lambda(U^-_{il})$ with $U^+_{cnl}\neq U^+_{if}$,
	
	\item[4a)] If the point in step $(4)$ does not exist then neither does it exist the non-local composite wave,
	
	\item[4b)] If the point in step $(4)$ does exist and $det(A(U^+_{cnl},U^-_{inl})) \neq 0$ then calculate the non-local composite curve solving the ODE \eqref{RH3}-\eqref{RH4} starting at the initial point $(U^+_{cnl},U^-_{inl})$,
	with initial direction $r^o_i$ such that $\triangledown \lambda_i \cdot r_i < 0$ and
	the initial direction of composite field  is the right hide side
	of \eqref{RH3} evaluated at point $(U^+_{cnl},U^-_{inl})$,
	
	\item[4c)] If the point in step $(4)$ does exist and $det(A(U^+_{cnl},U^-_{inl})) = 0$ then calculate the non-local composite curve solving the ODE \eqref{RH3}-\eqref{RH4} starting at the initial point $(U^+_{cnl}+\epsilon_1 V_c,U^-_{inl})$,
		where $V_c$ is the value of initial composite field but in the direction where $det(A(U^+_{cnl}+\epsilon_1 V_c,U^-_{inl}))<0$.
\end{itemize}

From now, we give a theoretical justification of parts of the above algorithm and some 
commentaries about numeric implementation of the composite curve.

When the rarefaction curves of the family $i$ arrive at an inflection locus or the boundary of state space, we need the ODE solver to stop automatically. To do so, we take a plane parallel to the surface that one wants to stop. The integration is continued until the distance to such plane is small enough.

Since the composite curve is the concatenation of a rarefaction with a characteristic shock, the initial point of the rarefaction is exactly the intersection of the extension curve of the secondary bifurcation curve with the rarefaction, i.e. $\sigma(U^-_{il},U^+_{if})=\lambda(U^-_{il})$.

In Algorithm \textbf{4.1} we use inequality $det(A(U^+_{cnl}+\epsilon_1 V_c,U^-_{inl}))<0$ in order to choose the correct direction, since the determinant is negative for Lax admissible shocks.

%el tema de determinante negativo
%el campo en el punto inicial de la compuesta, pois os campos estan acoplados,
%y la refaccacion va en el sentido contrario.
%la direccion inicial tiene que ser la correcta.

\begin{figure}
	\centering
	\includegraphics[width=0.5\textwidth]{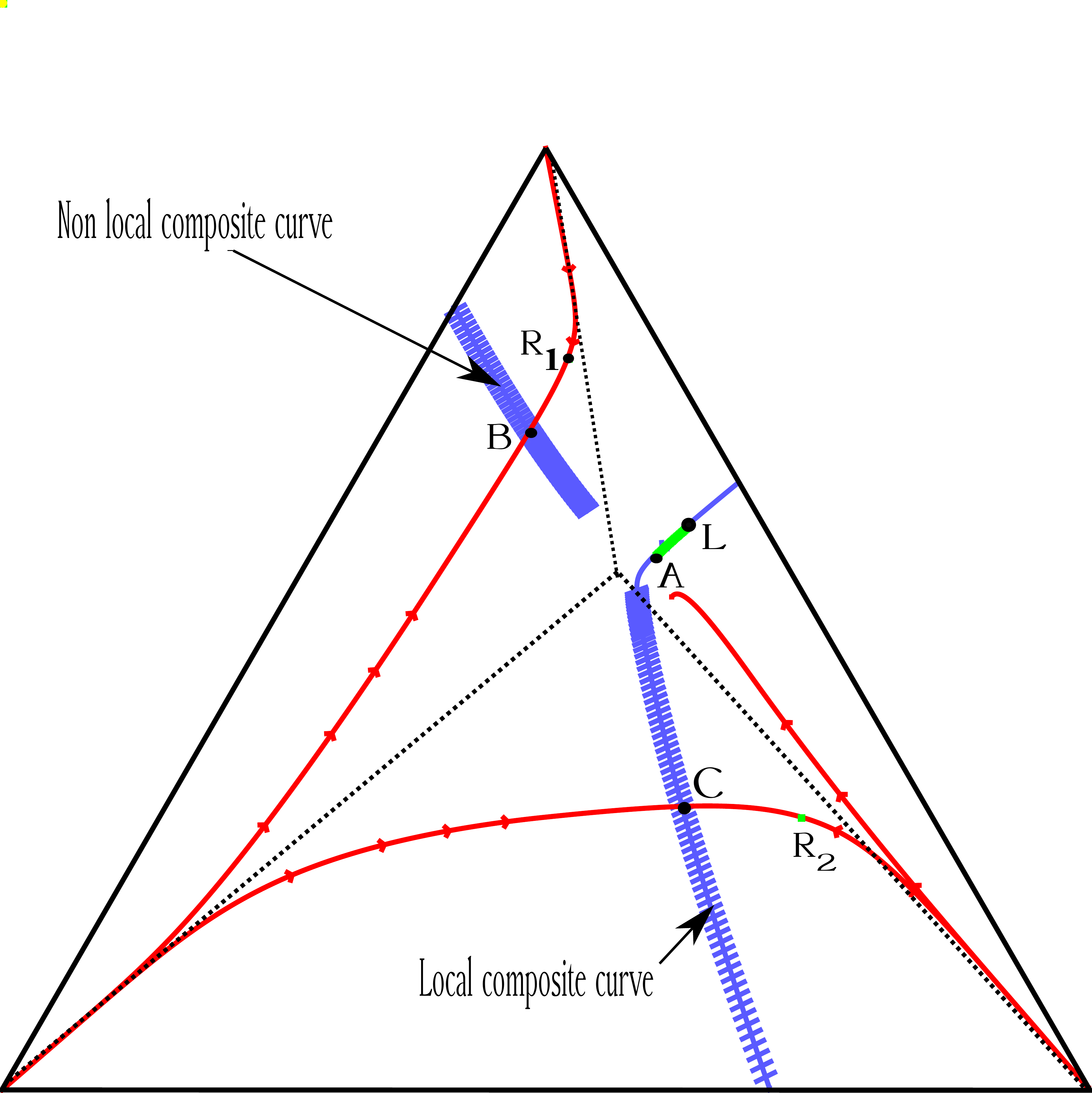}
	\caption{Triangular geometry of the space phases. Red curves represent backward shock wave curves from points $R_1$ and $R_2$. Green color curve represents a rarefaction curve from $L$ to $A$. Cyan curve represents the local and non-local composite wave curve corresponding to the rarefaction wave curve from $L$ to $A$. Dotted black curve represent the inflection locus which coincides in the unique umbilic point.}
	\label{figF1}
\end{figure}

\begin{figure}
	\centering
	\includegraphics[width=0.55\textwidth]{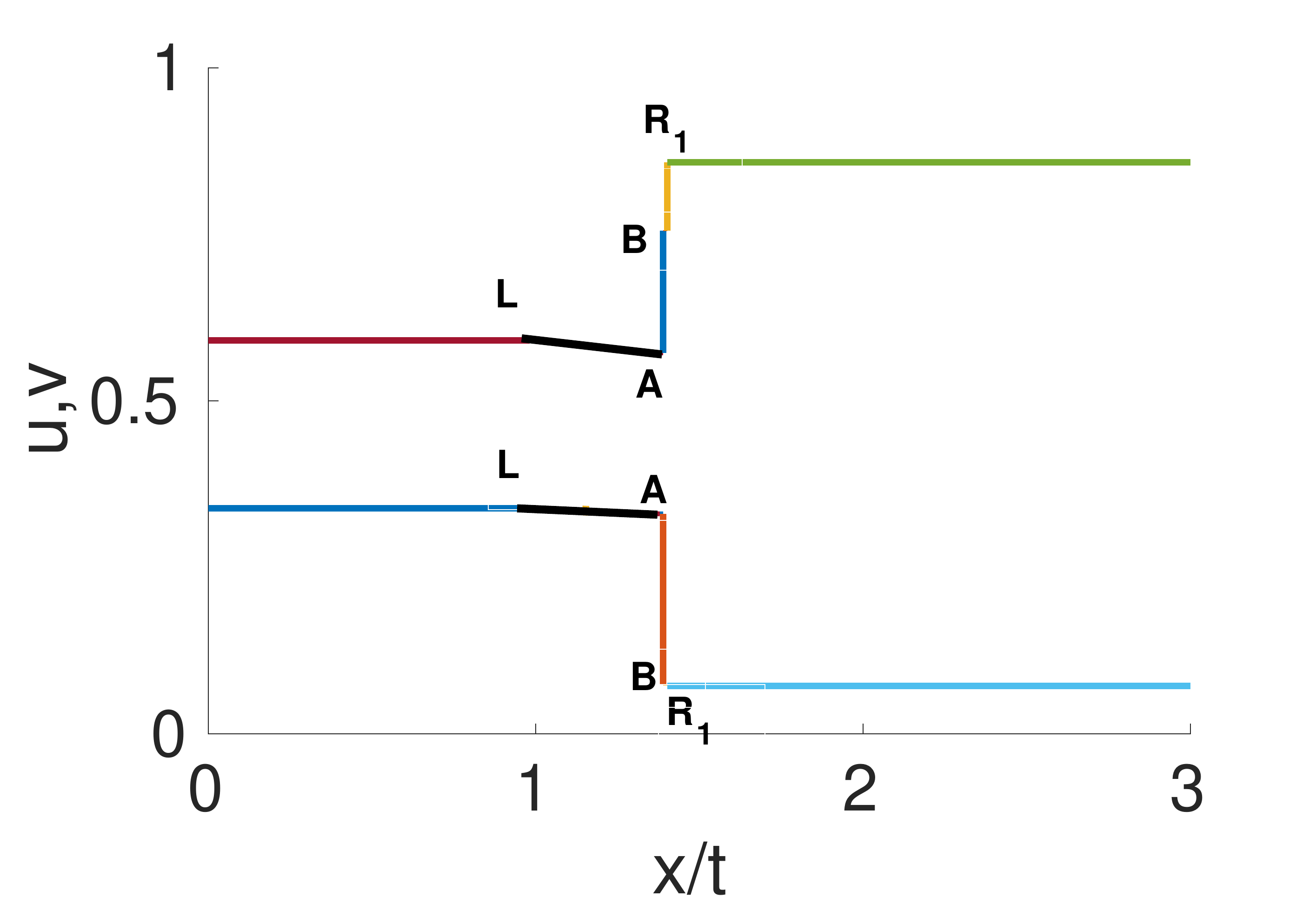}
	\caption{Profile functions corresponding with the wave curve in the space phases. The first corresponds to the wave curve with non-local while the second to the local composite wave curve.}
	\label{figFa1}
\end{figure}

\section{Wave curve construction}
	
A general numerical procedure for the construction of wave curves in the state space consists of several factors:
\begin{itemize}
	\item a set of input suitable parameters for the ODE solver,
	\item procedures to find the bifurcation curve (e.g. coincidence and inflection loci) 
	\item criteria to either stop integration or change to an a appropriate continuation algorithm beyond these intersection points,
	\item admissibility criteria for shock and rarefaction wave curves,
	\item check for monotonicity of characteristic velocities and stopping criteria for the solver.
\end{itemize}

\subsection{Starting point}
	
Of course, the characteristic field must be well defined and point to an admissible direction at the initial point.
	
The algorithm has peculiar behavior when the starting point $U^-$ lies on the coincidence or inflection loci. First, it explores the neighborhood of the initial point to discover the directions on which eigenvalues increase or decrease. Assuming that eigenvalues increase along forward rarefactions and decrease along backward rarefactions, there are three scenarios of rarefaction wave curves passing through point $U^-$: two forward rarefactions; one forward and one backward; or two backward rarefactions. Admissible shocks are also possible in the direction contrary to the rarefaction.

\subsection{Intermediate states}

ODE solvers are used for the continuation of wave curves. They yield correct result providing the appropriate fields for that wave curve segment (be it a rarefaction, a composite or a shock) are well defined at every passing state. The main result of this work is the reformulation of fields for the removal of singularities in a number of relevant scenarios, thereby defining appropriate procedures for the continuation of wave curves. 
	
A Riemann solution is obtained as a sequence of concatenated wave curves $\mathcal{W}_i : \Re \rightarrow \Re^n$, $i=1,\cdots,m$, parametrized by $\xi_i$ and continued from initial state $\mathcal{W}_i(0) = U_i^-$ which must also be the last point of the previous curve $\mathcal{W}_{i-1}$. The specific field used during the continuation of each wave curve $\mathcal{W}_i$ depends on wave type (e.g. rarefaction, composite or shock wave curves) and admissibility criteria must be verified during the construction of each wave curve segment.

The solution of a Riemann problem from state $U_L$ to $U_R$ is given by a chain of concatenated wave curves represented as by
\begin{equation}
	 \mathcal{W}_m(\cdots,\mathcal{W}_2(\mathcal{W}_1(U_L,\xi_1),\xi_2),\cdots,\xi_m)=U_R,
\end{equation} 
where $\xi_1, \xi_2, \dots, \xi_m$ are the parameters for wave curve segments and each wave after the second starts at some intermediate state $U_i^-$ which is the end of the previous curve in the sequence. 

The heart of the problem lies on finding the states where each curve segment ends and determining the appropriate numerical procedure for the continuing on with the next wave curve.

\subsection{Stopping criteria}

Integration along an integral curve stops only when it reaches either a physical boundary or some the bifurcation curve. Detection algorithms for both these situations have been described in detail along Sections \ref{coin1}, \ref{plane} and \ref{subsect:stop_inflection}.

\section{Conclusion}
\label{conclu}

	We introduced a change of variables, based on a generalized Jordan chain, in order to analyze resonance phenomena in systems of conservation laws. Based on this analysis, we propose a procedure for continuing wave curves beyond points where two characteristic speeds coincide. The continuation method for constructing wave curves is improve by solving numerical difficulties at some singularities and we present a theoretical argument for the existence of wave curves after their passage through a coincidence locus. One of the fundamental contributions is a proposal on how to take into account	the derivatives of flow and accumulation functions to better define the step for the integrator to advance in the neighborhood of points where resonance happens. 
	
	We prove the existence of composite wave curve when it traverse either the inflection locus or an anomalous part of the non-local composite wave curve. Another situation, regarding the construction of composite wave curves in anomalous points is left as for future works.
	
	The theoretical and numerical difficulties studied in this work, as well as their solutions, are illustrated by a series of relevant examples.\\

\textbf{Acknowledgements}

Special thanks to Teresa Braga de Queiroz are due for \gui{the preparation of}
some figures presented in this work. The authors \gui{are grateful to Professors} Bradley Plohr (Los Alamos National Laboratory) and Frederico Furtado (University of Wyoming) for key recommendations\gui{, and also thank the key} contributors to the RPN program: Edson Lima, Bradley Plohr and Rodrigo Morante. 

The first author's work was supported in part by IMPA/CAPES, while the second author work was supported in part by \gui{FAPERJ under Grant 202.574/2016}. 
The third author  was supported in part FINEP 01.13.0390.00/01.CNPq  304264/2014-8.
FAPERJ E-26/ 202.764/2017 cientista. E-26/010.001267/2016 Pronex.

\appendix

\section{Special vectors at the coincidence locus}
\label{apend1}

Equation \eqref{unoa1} can rewritten as

\begin{align}
\label{df1}
A(U)R_o(U)&=(\lambda_o+s(U))B(U)R_o(U)+p(U)B(U)R_1(U),\\
\label{df2}
A(U)R_1(U)&=(\lambda_o+s(U))B(U)R_1(U)+B(U)R_o(U).
\end{align}

At a point $U_o$ in the intersection surface $\left\lbrace U: \lambda_i(U)=\lambda_{i+1}(U) \right\rbrace$, we have $s(U_o)=p(U_o)=0$, $R_o(U_o)=r_o$ and
$R_1(U_o)=r_1$. We use the shortened notation $A(U_o)=A_o$ and $B(U_o)=B_o$.

Differentiating \eqref{df1} and \eqref{df2} with respect to $U_k$ at point $U_o$, we obtain
\begin{equation}
(A_o-\lambda_oB_o)\frac{\partial R_o}{\partial U_k}=\frac{\partial s}{\partial U_k}B_or_o+\frac{\partial p}{\partial U_k}B_or_1+\lambda_o\frac{\partial B}{\partial U_k}r_o-\frac{\partial A}{\partial U_k}r_o,
\label{df4}
\end{equation}
\begin{equation}
(A_o-\lambda_oB_o)\frac{\partial R_1}{\partial U_k}=\frac{\partial B}{\partial U_k}r_o+B_o\frac{\partial R_o}{\partial U_k}+\frac{\partial s}{\partial U_k}B_or_1-\frac{\partial A}{\partial U_k}r_1+\lambda_o\frac{\partial B}{\partial U_k}r_1.
\label{df5}
\end{equation}
Using the generalized Jordan chain in equations \eqref{gen4}-\eqref{gen3a1}, multiplying \eqref{df4} by left eigenvector $l_o$ and using that  $l_o(A_o-\lambda_oB_o)=0$ we obtain
\begin{equation}
\frac{\partial p}{\partial U_k}=l_o\frac{\partial A}{\partial U_k}r_o-\lambda_o l_o\frac{\partial B}{\partial U_k}r_o.
\label{chanap}
\end{equation}
The derivative $\frac{\partial s}{\partial U_k}$ is calculated by taking the sum of equations \eqref{df4} and \eqref{df5} pre-multiplied by $l_1$ and $l_o$, respectively and using identities \eqref{gen4}-\eqref{gen3a1} as
\begin{equation}
\frac{\partial s}{\partial U_k}=\frac{1}{2}\left(l_o\frac{\partial A}{\partial U_k}r_1+l_1 \frac{\partial A}{\partial U_k}r_o\right)-\frac{\lambda_o}{2}\left(l_o\frac{\partial B}{\partial U_k}r_1+l_1 \frac{\partial B}{\partial U_k}r_o\right).
\end{equation} 
Explicit formulas for  $\frac{\partial R_o}{\partial U_k}$ and $\frac{\partial R_1}{\partial U_k}$ can be obtained from equations \eqref{df4} and \eqref{df5}.
In order to solve these equations to obtain the derivatives of $R_0$ and $R_1$, we first define
\begin{equation}
Z=A_o-\lambda_oB_o+B_or_1l_1B_o, 
\label{ApA:def_G}
\end{equation}
such that for a given $b$, there is a unique $x=Z^{-1}b$ satisfying $l_1B_ox=0$ and $(A_o-\lambda_oB_o)x=b$ (see \cite{seyranian2003multiparameter} ).
Moreover, from \eqref{gen4}-\eqref{gen3a1} we have
\begin{equation}
Zr_o=Br_1, \quad Zr_1=Br_o, \quad l_oZ=l_1B_o \quad l_1Z=l_oB_o.
\end{equation}
Then, from equations \eqref{df4} and \eqref{df5} we obtain
\begin{equation}
\frac{\partial R_o}{\partial U_k}=\frac{\partial s}{\partial U_k}r_1+\frac{\partial p}{\partial U_k}r_o+Z^{-1}\left(\lambda_o\frac{\partial B}{\partial U_k}r_o-\frac{\partial A}{\partial U_k}r_o\right),
\label{df4a1}
\end{equation}
\begin{equation}
\frac{\partial R_1}{\partial U_k}=\frac{\partial s}{\partial U_k}r_o+Z^{-1}\left (\frac{\partial B}{\partial U_k}r_o+B_o\frac{\partial R_o}{\partial U_k}-\frac{\partial A}{\partial U_k}r_1+\lambda_o\frac{\partial B}{\partial U_k}r_1\right).
\label{df5b11}
\end{equation}

Finally, $R_o(U)$ and $R_1(U)$ can be approximated using Taylor's formula to first order

	\begin{equation}
	\begin{split}
	& R_o(U) = r_o+\sum_{k=1}^{n}\frac{\partial R_o}{\partial U_k}(U^k-U_o^k) + o(||U-U_o||^2), \\
	& R_1(U) = r_1+\sum_{k=1}^{n}\frac{\partial R_1}{\partial U_k}(U^k-U_o^k) + o(||U-U_o||^2),
	\end{split}
	\end{equation}
where $U=(U^1,\dots,U^n)$ and $U_o=(U_{o}^1,\dots,U_{o}^n)$.

%A formula for the derivative of an eigenvector can be found using the normalization condition $l_0^T B_0 r = const$, which implies that
%\begin{equation} \label{ApB:derivative_of_normalization}
%l_0^T B_0 \frac{\partial r}{\partial U^-_j} = 0.
%\end{equation}
%Let $\bar{l_o}$ be the complex conjugate of $l_o$. After multiplying \eqref{ApB:derivative_of_normalization} to the left by $\bar{l}$ into \eqref{ApB:eigenvalue_derivative_temp1}

%formulas Delta 2, y derivada deo autovalores deben ser dadas

\bibliographystyle{plain}
  
\bibliography{resonancia2}

\begin{thebibliography}{10}

\bibitem{alvarez2017analytical1}
A.~C. Alvarez, T~Blom, W.~J. Lambert, J.~Bruining, and D.~Marchesin.
\newblock Analytical and numerical validation of a model for flooding by saline
  carbonated water.
\newblock {\em Journal of Petroleum Science and Engineering}, 167:900--917,
  2018.

\bibitem{alvarez2017analytical}
AC~Alvarez, J~Bruining, WJ~Lambert, and D~Marchesin.
\newblock Analytical and numerical solutions for carbonated waterflooding.
\newblock {\em Computational Geosciences}, 22(2):505--526, 2018.

\bibitem{ancona2001note}
Fabio Ancona and Andrea Marson.
\newblock A note on the {R}iemann problem for general nxn conservation laws.
\newblock {\em Journal of mathematical analysis and applications},
  260(1):279--293, 2001.

\bibitem{arnold2012geometrical}
V.~I. Arnold.
\newblock {\em Geometrical {M}ethods in the {T}heory of {O}rdinary
  {D}ifferential {E}quations}, volume 250.
\newblock Springer Science \& Business Media, 2012.

\bibitem{azevedo1995multiple}
A.~V. Azevedo and D.~Marchesin.
\newblock Multiple viscous solutions for systems of conservation laws.
\newblock {\em Transactions of the American Mathematical Society},
  347(8):3061--3077, 1995.

\bibitem{ben2003generalized}
Adi Ben-Israel and Thomas~NE Greville.
\newblock {\em Generalized {I}nverses: {T}heory and {A}pplications}, volume~15.
\newblock Springer Science \& Business Media, 2003.

\bibitem{dahmen2005riemann}
W.~Dahmen, S.~M{\"u}ller, and A.~Vo{\ss}.
\newblock Riemann problem for the euler equation with non-convex equation of
  state including phase transitions.
\newblock In {\em Analysis and Numerics for Conservation Laws}, pages 137--162.
  Springer, 2005.

\bibitem{datta2010numerical}
B.~N. Datta.
\newblock {\em Numerical {L}inear {A}lgebra and {A}pplications}, volume 116.
\newblock SIAM, 2010.

\bibitem{elden1982weighted}
L.~Eld{\'e}n.
\newblock A weighted pseudoinverse, generalized singular values, and
  constrained least squares problems.
\newblock {\em BIT Numerical Mathematics}, 22(4):487--502, 1982.

\bibitem{furtado1991structural}
F.~Furtado.
\newblock {\em Structural stability of nonlinear waves for conservation laws}.
\newblock PhD thesis, New York Univ., 1989.

\bibitem{gantmacher1960theory}
F.~R. Gantmacher.
\newblock {\em The {T}heory of {M}atrices. 1 (1960)}.
\newblock Chelsea, 1960.

\bibitem{gentle2007matrix}
J.~E. Gentle.
\newblock {\em Matrix {A}lgebra: {T}heory, {C}omputations, and {A}pplications
  in {S}tatistics}.
\newblock Springer Science \& Business Media, 2007.

\bibitem{golub1996matrix}
G.~H. Golub and F.~Van~Loan, Charles.
\newblock {M}atrix {C}omputations, 3rd, 1996.

\bibitem{golubitsky2012singularities}
Martin Golubitsky, Ian Stewart, and David~G Schaeffer.
\newblock {\em Singularities and {G}roups in {B}ifurcation {T}heory}, volume~2.
\newblock Springer Science \& Business Media, 2012.

\bibitem{hansen1998rank}
P.~C. Hansen.
\newblock {\em Rank-{D}eficient and {D}iscrete {I}ll-{P}osed {P}roblems:
  {N}umerical {A}spects of {L}inear {I}nversion}.
\newblock SIAM, 1998.

\bibitem{helmut2011thermal}
W.~Helmut.
\newblock {\em Thermal Effects in the Injection of CO2 in Deep Underground
  Aquifers}.
\newblock PhD thesis, IMPA. Brazil, 2011.

\bibitem{issacson1992global}
E.~L. Issacson, D.~Marchesin, C.~F. Palmeira, and J.~Plohr, Bradley.
\newblock A global formalism for nonlinear waves in conservation laws.
\newblock {\em Communications in mathematical physics}, 146(3):505--552, 1992.

\bibitem{key95x}
B.~L. Keyfitz.
\newblock A geometric theory of conservation laws which change type.
\newblock {\em Z. Angew. Math. Mech.}, 75:571--581, 1995.

\bibitem{lambert2010riemann}
W.~Lambert, D.~Marchesin, and J.~Bruining.
\newblock The {R}iemann solution for the injection of steam and nitrogen in a
  porous medium.
\newblock {\em Transport in porous media}, 81(3):505--526, 2010.

\bibitem{lambert2006riemann}
W.~J. Lambert.
\newblock {\em Riemann solutions of balance system with phase change for
  thermal flow in porous media}.
\newblock PhD thesis, IMPA, 2006.

\bibitem{lax1957hyperbolic}
P.~D. Lax.
\newblock Hyperbolic systems of conservation laws ii.
\newblock {\em Communications on pure and applied mathematics}, 10(4):537--566,
  1957.

\bibitem{liu1974riemann}
T.~P. Liu.
\newblock The {R}iemann problem for general 2$\times$ 2 conservation laws.
\newblock {\em Transactions of the American Mathematical Society}, 199:89--112,
  1974.

\bibitem{liu1975riemann}
T.~P. Liu.
\newblock The {R}iemann problem for general systems of conservation laws.
\newblock {\em Journal of Differential Equations}, 18(1):218--234, 1975.

\bibitem{mailybaev2000transformation}
A.~A. Mailybaev.
\newblock Transformation of families of matrices to normal forms and its
  application to stability theory.
\newblock {\em SIAM Journal on Matrix Analysis and Applications},
  21(2):396--417, 2000.

\bibitem{mailybaev2001transformation}
A.~A. Mailybaev.
\newblock Transformation to versal deformations of matrices.
\newblock {\em Linear Algebra and its Applications}, 337(1-3):87--108, 2001.

\bibitem{mailybaev2006computation}
A.~A. Mailybaev.
\newblock Computation of multiple eigenvalues and generalized eigenvectors for
  matrices dependent on parameters.
\newblock {\em Numerical Linear Algebra with Applications}, 13(5):419--436,
  2006.

\bibitem{mailybaev2008hyperbolicity}
A.~A. Mailybaev and D.~Marchesin.
\newblock Hyperbolicity singularities in rarefaction waves.
\newblock {\em Journal of Dynamics and Differential Equations}, 20(1):1--29,
  2008.

\bibitem{mailybaev2008lax}
A.~A. Mailybaev and D.~Marchesin.
\newblock Lax shocks in mixed-type systems of conservation laws.
\newblock {\em Journal of Hyperbolic Differential Equations}, 5(02):295--315,
  2008.

\bibitem{matos2015bifurcation}
V.~Matos, A.~V. Azevedo, J.~C. Da~Mota, and D.~Marchesin.
\newblock Bifurcation under parameter change of riemann solutions for
  nonstrictly hyperbolic systems.
\newblock {\em Zeitschrift f{\"u}r angewandte Mathematik und Physik},
  66(4):1413--1452, 2015.

\bibitem{Matos2015}
V.~Matos, A.~V. Azevedo, J.~C. Da~Mota, and D.~Marchesin.
\newblock Bifurcation under parameter change of riemann solutions for
  nonstrictly hyperbolic systems.
\newblock {\em Zeitschrift f{\"u}r angewandte Mathematik und Physik},
  66(4):1413--1452, Aug 2015.

\bibitem{matos2017compositional}
V{\'\i}tor Matos and Dan Marchesin.
\newblock Compositional flow in porous media: Riemann problem for three
  alkanes.
\newblock {\em Quarterly of Applied Mathematics}, 75(4):737--767, 2017.

\bibitem{Matos2016}
Vitor Matos, Julio~D. Silva, and Dan Marchesin.
\newblock Loss of hyperbolicity changes the number of wave groups in riemann
  problems.
\newblock {\em Bulletin of the Brazilian Mathematical Society, New Series},
  47(2):545--559, Jun 2016.

\bibitem{moler1997we}
C.~Moler.
\newblock Are we there yet?
\newblock {\em Zero crossing and event handling for differential equations,
  Matlab News \& Notes}, pages 16--17, 1997.

\bibitem{muller2001existence}
S.~M{\"u}ller and A.~Voss.
\newblock {\em On the existence of the composite curve near a degeneration
  point}.
\newblock RWTH Aachen. Institut f{\"u}r Geometrie und Praktische Mathematik,
  2001.

\bibitem{newman1968numerical}
J.~Newman.
\newblock Numerical solution of coupled, ordinary differential equations.
\newblock {\em Industrial \& Engineering Chemistry Fundamentals},
  7(3):514--517, 1968.

\bibitem{palmeira1988line}
C.~F. Palmeira.
\newblock Line fields defined by eigenspaces of derivatives of maps from the
  plane to itself.
\newblock In {\em Proceedings of the VIth International Conference of
  Differential Geometry, Santiago de Compostela, Spain}, pages 177--205, 1988.

\bibitem{seyranian2003multiparameter}
A.~P. Seyranian and A.~A. Mailybaev.
\newblock {\em Multiparameter {S}tability {T}heory with {M}echanical
  {A}pplications}, volume~13.
\newblock World Scientific, 2003.

\bibitem{shampine1997matlab}
L.~F. Shampine and M.~W. Reichelt.
\newblock The {M}atlab {ODE} suite.
\newblock {\em {SIAM} journal on scientific computing}, 18(1):1--22, 1997.

\bibitem{shampine2000event}
L.~F. Shampine and S.~Thompson.
\newblock Event location for ordinary differential equations.
\newblock {\em Computers \& Mathematics with Applications}, 39(5-6):43--54,
  2000.

\bibitem{Shearer89}
M.~Shearer.
\newblock The {R}iemann problem for 2 x 2 systems of hyperbolic conservation
  laws with case {I} quadratic nonlinearities.
\newblock {\em J. Differential Equations}, pages 343--363, 1989.

\bibitem{wendroff1972riemann}
B.~Wendroff.
\newblock The {R}iemann problem for materials with nonconvex equations of state
  {I}: Isentropic flow.
\newblock {\em Journal of Mathematical Analysis and Applications},
  38(2):454--466, 1972.

\bibitem{wendroff1972riemann2}
B.~Wendroff.
\newblock The {R}iemann problem for materials with nonconvex equations of state
  {II}: General flow.
\newblock {\em Journal of Mathematical Analysis and Applications},
  38(3):640--658, 1972.

\end{thebibliography}

\end{document}